\documentclass[10pt]{amsart}
\usepackage{amsmath,amssymb,verbatim}
\usepackage[utf8]{inputenc}
\usepackage{enumerate}
\usepackage{pdflscape}
\usepackage{amsthm}
\usepackage{mathrsfs}
\usepackage{color}
\usepackage[normalem]{ulem}
\usepackage{cancel}
\usepackage{tikz}
\usetikzlibrary{matrix}
\usepackage[all]{xy}
\usepackage{mathtools}
\usepackage{cleveref}
\usepackage{arydshln}
\usepackage[shortlabels]{enumitem}

\topskip=\baselineskip 
\newtheorem{theorem}{Theorem}[section]
\newtheorem{lemma}[theorem]{Lemma}
\newtheorem{proposition}[theorem]{Proposition}

\newtheorem{corollary}[theorem]{Corollary}
\newtheorem{remark}[theorem]{Remark}

\newtheorem{problem}[theorem]{Problem}

\makeatletter
\@addtoreset{equation}{section}
\makeatother

\newcommand{\M}{{\mathrm D}}

\newcommand{\Imagen}{\mbox{\rm Im }}

\newcommand{\R}{{\mathbb R}}

\newcommand{\F}{{\mathbb F}}

\newcommand{\matriz}[1]{\begin{array} #1 \end{array}}

\newcommand{\GEN}[1]{\left\langle #1 \right\rangle}

\newcommand{\ZZ}{\mathcal{Z}}

\newcommand{\qand}{\quad \text{and} \quad}

\title[On the Modular Isomorphism Problem for groups of nilpotency class $2$ with cyclic center]{On the Modular Isomorphism Problem for groups of nilpotency class $2$ with cyclic center}

\author{ Diego Garc\'{\i}a-Lucas and Leo Margolis}


\begin{document}
	\maketitle
	
	 \begin{abstract}
	 We show that the modular isomorphism problem has a positive answer for groups of nilpotency class $2$ with cyclic center, i.e. that for such $p$-groups $G$ and $H$ an isomorphism between the group algebras $FG$ and $FH$ implies an isomorphism of the groups $G$ and $H$ for $F$ the field of $p$ elements. For groups of odd order this implication is also proven for $F$ being any field of characteristic $p$. For groups of even order we need either to make an additional assumption on the groups or on the field.
\end{abstract}
	
	\section{Introduction and main result}
	
	Let $p$ be a prime, $\F_p$ the field with $p$ elements, $G$ and $H$ finite $p$-groups and denote by $FG$ the group algebra of $G$ over a field $F$. The modular isomorphism problem asks whether the existence of a ring isomorphism $\F_pG\cong \F_pH$ implies the existence of an isomorphism of groups $G\cong H$. In particular, it appeared in Brauer's surevy from 1963 \cite{Bra63}. 
%	Bla bla bla. This generalizes the result of Drensky \cite{Drensky}, in which the modular isomorphism problem for every field of characteristic $p$ is solved for groups with class $2$ and center of index 
The modular isomorphism problem received some attention during the following decades, but was solved positively only for special classes of groups, remaining poorly understood in general, as evident from the account given in a beautiful survey of Sandling \cite{Sandling85}. Since then some new techniques were tried, e.g. in \cite{San89}, including computer aided investigations \cite{Wursthorn1993, Eick08}, and more positive results achieved, cf. \cite{Mar22} for an overview. Recently the problem was solved in its generality by the discovery of a series of non-isomorphic $2$-groups which have isomorphic group algebras over any field of characteristic $2$ \cite{GarciaMargolisdelRio}. Nevertheless, the original problem as well as variations of it remain open and interesting. A very obvious class of groups are those of odd order. Another interesting class is that of groups of nilpotency class $2$, about which Sandling in his 1985 survey \cite{Sandling85} writes: 

	\begin{quotation}
		\textit{Nonetheless, it is a sad reflection on the state of the modular isomorphism problem that the case of class 2 groups is yet to be decided in general.}
	\end{quotation}   
	
Note that the examples exhibited in \cite{GarciaMargolisdelRio} have nilpotency class $3$, so that the groups of class $2$ are situated between the abelian groups, for which the first positive result on the modular isomorphism problem was proven \cite{Deskins1956}, and a class known to contain negative solutions. A further interesting aspect of the problem concerns the underlying field: could it happen that $FG \cong FH$ for some field $F$ of characteristic $p$, but $kG \not\cong kH$ for some other field $k$ of the same characteristic? No answer to this question is known (see also \cite{GarciLucasDelRio23}).

We give a quick account of all the known results for groups of nilpotency class $2$: Passi and Sehgal showed that the problem has a positive answer for groups of class $2$ and exponent $p$ or $4$ over the prime field \cite{PS72} and Sandling showed the same for groups of class $2$ with elementary abelian derived subgroup \cite{San89}. Drensky studied the class of groups with center of index $p^2$ for which he solved the problem independently of the underlying field \cite{Drensky}, generalizing a result of Nachev and Mollov \cite{NachevMollov}. Rather recently it was solved, again over the prime field, for $2$-generated groups of class $2$ \cite{BdR20}. Taking into account that the results of \cite{San89} are already mentioned in Sandling's survey, one can say, that the situation did not change much up to now concerning the quote given above.
	
As an initial attempt to solve the problem for the wide class of groups of nilpotency class $2$, we solve the modular isomorphism problem positively for those groups of class $2$ which have cyclic center. Moreover, we pay close attention to the underlying field and exhibit interesting behavior in this regard for the groups of even order in our class. Our results read as follows.

\begin{theorem}\label{TheoremMIPCyclicCentreOdd} 
Let $p$ be an odd prime, $F$ a field of characteristic $p$ and $G$ a finite $p$-group of nilpotency class 2 with cyclic center. If $H$ is a group such that the modular group algebras $FG$ and $FH$ are isomorphic, then the groups $G$ and $H$ are isomorphic. 
\end{theorem}

\begin{theorem}\label{TheoremMIPCyclicCentreEvenGoodFieds} 
Let $F$ be a field of characteristic $2$ and $G$ a finite $p$-group of nilpotency class 2 with cyclic center. Assume moreover that the polynomial $X^2+X+1$ is irreducible in the polynomial ring $F[X]$. If $H$ is a group such that the modular group algebras $FG$ and $FH$ are isomorphic, then the groups $G$ and $H$ are isomorphic. 
\end{theorem}

Over arbitrary fields of characteristic $2$ we need a restriction on the groups involved. For $G$ a $2$-group of class $2$, let $m(G)$ be the rank of the homocylic component of $G/Z(G)$ of maximal exponent, i.e. the number of cyclic direct factors of maximal order in this group. This is the same as the rank of the elementary abelian group $ G/ \Omega_{\log_2(\exp(G))-1}(G)$, where $\Omega_n(G)$ denotes the subgroup of $G$ generated by elements of order at most $2^n$.

%\begin{theorem}\label{TheoremMIPCyclicCentreEvenBadFields} 
%Let $F$ be a field of characteristic $2$ and $G$ a finite $2$-group of nilpotency class 2 with cyclic center. Denote by $s$ be the rank of the elementary abelian $2$-group \ChDiego{\sout{ $G/(G^{\exp(G)/2})$.} $ G/ \Omega_{\log_2(\exp(G))-1}(G)$.} \ChLeo{I understand you want to make the notation the same everywhere, but if we want to use $\Omega$ here, we need a to put the index $\log(\exp(G)/2)$ which seems too much}\ChDiego{[You are right about the missing $\log_2$. I've changed it. Yet, $G^{2^?}$ is $\mho_?(G)$, while what I think we really want is $\Omega_?(G)$: otherwise it is not true that $G$ modulo this subgroup is elementary abelian.]} If $FG$ and $FH$ are isomorphic and $s \leq 2$, then $G$ and $H$ are isomorphic. If $FG$ and $FH$ are isomorphic and $s > 2$, then there is exactly one isomorphism  type of groups not containing $G$ to which $H$ possibly belongs. 
%
%The same statements hold when $F$ is finite for $s \leq 4$ and $s > 4$, respectively.
%\end{theorem}

\begin{theorem}\label{TheoremMIPCyclicCentreEvenBadFields} 
Let $F$ be a field of characteristic $2$ and $G$ a finite $2$-group of nilpotency class 2 with cyclic center. If $FG$ and $FH$ are isomorphic and $m(G) \leq 2$, then $G$ and $H$ are isomorphic. If $FG$ and $FH$ are isomorphic and $m(G) > 2$, then there is exactly one isomorphism  type of groups not containing $G$ to which $H$ possibly belongs. 
\end{theorem}

The isomorphism types of the $2$-groups for which we can not give a complete answer over arbitrary fields is described in Lemma~\ref{lemma:LastCaseStanding}. The smallest of those have order $2^{12}$.

All our proofs rely on the explicit description of the groups of interest by Leong \cite{Leong1974, Leong1979}. The proof of Theorem~\ref{TheoremMIPCyclicCentreOdd} relies mostly on the analysis of well-known properties of $G$ which are known to be determined by the isomorphism type of $FG$, though we also provide a general lemma which can be useful for calculations in modular group algebras of $p$-groups of nilpotency class at most $p-1$ (Lemma~\ref{lem:FrobeniusMap}). The properties used for the groups of odd order turn out not to be sufficient when groups of even order are considered. A standard idea in this situation is to define certain maps on quotients of the group algebra and count the elements satisfying certain properties, as suggested in \cite{Bra63} and carried out for the first time in \cite{Passman1965p4}. We carry out a similar argument using a power map, but as we work not only with finite fields, the counting process is replaced by the comparison of quadratic forms in characteristic $2$ based on classical work of Arf \cite{Arf}. This comparison quite naturally leads to the condition on $F$ appearing in Theorem~\ref{TheoremMIPCyclicCentreEvenGoodFieds}. The proof of Theorem~\ref{TheoremMIPCyclicCentreEvenBadFields} on the other hand relies on a direct, rather technical calculation in the group algebra. This calculation leads to the question, if the solutions of certain polynomials can be transferred one to the other under a bijective linear map. This question in its generality has an algebraic geometry flavor  or a combinatorial one, if one restricts to finite fields. We can answer the question negatively in the situations stated in Theorem~\ref{TheoremMIPCyclicCentreEvenBadFields}. 

We would like to remark that the underlying field has usually not played a significant role in the investigations of the Modular Isomorphism Problem. The methods applied in Sections \ref{sec:QuadForms} and \ref{sec:General2Groups} indicate though that this might lead to interesting new methods which have not been used so far to study modular group algebras of $p$-groups. 

\begin{remark} To complete our account on the knowledge for groups of nilpotency class $2$, we mention that though the result for $2$-generated groups of class $2$ in \cite{BdR20} only mentions the prime field, their arguments for groups of odd order work over any field. This is not the case though, for the groups of even order.
\end{remark}

\section{Notation and description of the groups}

\subsection{Notation and general results}
Throughout the whole article $p$ will denote a prime and $G$ will always be a finite $p$-group and $F$ a field of characteristic $p$. By the \emph{class} of $G$ we will always mean its nilpotency class. Our group theoretical notation is mostly standard: $\Phi(G)$ denote the Frattini subgroup of $G$, $\gamma_i(G)$ the $i$-th term of the lower central series of $G$, $\ZZ(G)$ the center of $G$, $C_n$ the cyclic group of order $n$ and a group of class $2$ is not abelian. We also write $\gamma_2(G)=G'$. Moreover, for a non-negative integer $i$ we denote $\Omega_i(G)=\GEN{g\in G: g^{p^i}=1}$,  $G^{p^i}=\GEN{g^{p^i}: g\in G}$ and $\Omega_i(G:N)=\GEN{g\in G: g^{p^i}\in N}$ for $N$ any normal subgroup of $G$. For elements $g,h\in G$, we denote by $[g,h]=g^{-1}h^{-1}gh$ the commutator of $g$ and $h$. 

We will write that a property or an invariant of the group $G$ is \emph{determined} if for any other group $H$ with $FG\cong FH$, then  $H$ also satisfies that property or agrees in the invariant. For  example, a classical result of Jennings \cite{Jen41} states that for $s\geq 1$ the isomorphism type of the quotient 
$$\M_{s}(G)/\M_{s+1} (G) $$
is determined, where $\M_{s}(G)$ is the $s$-th term of the Jennings series of $G$, namely: 
\begin{equation}\label{eq:JenningsSeries} 	
\M_{s}(G)=\prod_{ip^j\geq s} \gamma_i(G)^{p^j},
\end{equation} 
In particular $\M_1(G)=G$ and  $\M_2(G)=\Phi(G) $, and thus the minimal number of generators of $G$, which equals the $p$-rank of $G/\Phi(G) $, is determined. Other examples are the property of $G$ being of nilpotency class $2$ or not (see \cite{BK07}) or the isomorphism type of $\ZZ(G)$ (see \cite{Ward}).   

We denote $I(G)$ the augmentation ideal of $FG$, i.e., the ideal generated by the elements of the form $g-1$, with $g\in G$.
We say an ideal $J$ of $FG$ is \emph{canonical}, if it depends only on the algebra structure of $FG$, that is,  it does not depend  on the   considered basis. For example, the augmentation ideal is canonical, as it is the Jacobson radical of $FG$. Also the powers of $I(G)$ are canonical, and, also by the same result of Jennings \cite{Jen41}, for $s\geq 1$ one has that 
$$I(G)^s\cap (G-1)=\M_s(G)-1.  $$
Given a normal subgroup $N$ of $G$, we denote by $I(N)FG$   the (two-sided) ideal of $FG$ generated by the elements of the form $n-1$, with $n\in N$. Then $I(N)FG\cap (G-1)=N-1$, and $FG/I(N)FG\cong F\left[G/N\right]$. The ideal  $I(\ZZ(G)G')FG$   is also canonical, as it coincides with the ideal generated by the subspace of the Lie commutators $[FG,FG]$ together with $\ZZ(FG)\cap I(G)$, as indicated in \cite{Sandling85} or in \cite{BK07}.
For each $i\geq 1$, the ideal $I(\Omega_i(G:G'))FG$ is canonical (see \cite[Theorem 3.1]{SalimTesis}).
Moreover, as $I(\ZZ(G))FG$ is canonical, so is $I(\Omega_i(G:\ZZ(G)))FG$ for every non-negative integer $i$ by \cite[Lemma 3.6]{Diego22}. 

Given two groups $G$ and $H$, we denote by $G*_\varphi H$ the central product of $G$ and $H$ at $\varphi$, i.e. there is a homomorphism $\varphi: \mathcal{Z}(G) \rightarrow \mathcal{Z}(H)$ and $G*_\varphi H = G \times H/\langle z\varphi(z^{-1}) \ | \ z \in \mathcal{Z}(G) \rangle$. In our situation $G$ and $H$ will always be $p$-groups with cyclic center, so that there is a monomorphism between their centers. Note that if $\varphi$ and $\psi$ are monomorphisms between cyclic groups, then $G *_\varphi H \cong G*_\psi H$. This will be in fact the only case of interest to us and we will simply write $G*H$ in this case.

%Let $A$ a maximal abelian subgroup of $G$ such that $G$ admits an internal direct product decomposition $G=U \times A$, for some other subgroup $U\subseteq A$. The choices of $A$ and  $U$ are not unique, but by the Krull-Remak-Schmidt their isomorphism types are. Hence they are group invariants of $G$. By   \cite[Proposition B]{??}\ChDiego{[A small leap of faith for you.]}, the isomorphism type of $A$ is determined by $FG$. 
%
%
%
%\begin{lemma}
%	If $\ZZ(G)\subseteq \Phi(G)$ then $G$ does not admit abelian internal direct factors.
%\end{lemma}
%\begin{proof}
%	Suppose that $G$ admits a decomposition $G=U\times C$, where $C$ is cyclic. Let $c$ be a generator of $C$, and $u_1,\dots, u_n$ be a Burnside basis of $U$. Then $\{u_1,\dots, u_n,c\}$ is a system of generators of $G$ but $\{u_1,\dots, u_n\}$ is not, so that $c$ is not a non-generator, thus $c\notin \Phi(G)$. As $c\in \ZZ(G)$, the lemma follows. 
%\end{proof}

%Moreover, for each integer $t\geq 1$ we denote by 
% $\mathfrak q_t $ the number of conjugacy classes of maximal elementary abelian subgroups of $G$ of rank $t$. It is known in the literature of about the modular isomorphism problem as the \emph{Quillen parameter}, as it was  follows from a result of Quillen \cite{Qui71}  that this number depends only on the $F$-algebra structure of $FG$, i.e., is determined by $FG$.

\subsection{The groups}
Our proof relies on the following two results of Leong, which classify the groups of our target class up to isomorphism. 

  If $n ,r $ are integers such that $0\leq r \leq n$,   let  $Q(n,r)$ be the group  
 \begin{equation}\label{eq:DefQnr}
 Q (n,r)=\begin{cases}
 \GEN{   a,b \ \mid \	a^{p^{n}}=b^{p^{r}}=1 , \ a^{p^{n-r}}=[a,b]  } , & \text{if }2r \leq n ; \\
 \GEN{a,b \ \mid \	a^{p^{n}}=b^{p^{r}}=1, \ a^{p^{r}}=[a,b]^{p^{2r-n} }, \ [a,[a,b]]= [b,[a,b]]=1}, & \text{if } r \leq n <2r. 
 \end{cases}
 \end{equation}
Note that $Q(n,0)$ is the cyclic group of order $p^n$.

 If $p=2$, we consider also 
 \begin{equation}\label{eq:DefRn}
 R(n)=\GEN{a,b \ \mid \  a^{2^{n+1}}= b^{2^{n+1}}=1, a^{2^n}=[a,b]^{2^{n-1}}=b^{2^n}, [a,[a,b]]=[b,[a,b]]=1 }
 \end{equation} 
 for $n\geq 1$ some positive integer.  
We will work with central products of several groups of the form $Q(n,r)$. In that case we will write $Q_i(n,r)=\GEN{a_i,b_i}$ for a copy of $Q(n,r)$, with the same presentation for $a_i$ and $b_i$ as in \eqref{eq:DefQnr}, where  $i$ is some positive integer. The generators of $R(n)$ will always be denoted by $a$ and $b$ and assumed to satisfy the relations of \eqref{eq:DefRn}.

 \begin{theorem}[\cite{Leong1974}]  \label{TheoremClassification}
 	Let  $p$ be an odd prime. If  $G$ is a non-trivial finite $p$-group with cyclic center and nilpotency class at most $2$, then $G$ is isomorphic to a group of the form:
 	\begin{equation}\label{GroupClassification}
 	Q_1(n_1,r_1)*\dots * Q_\alpha(n_\alpha,r_\alpha) * Q_{\alpha+1}(\ell_1,\ell_1) *\dots  *Q_{\alpha+\beta } (\ell_\beta,\ell_\beta)  ,
 	\end{equation} 
 	for some non-negative   integers $\alpha,\beta,n_1,r_1,\dots, n_\alpha, r_\alpha, \ell_1,\dots ,\ell_\beta $, satisfying the following conditions: 
 	\begin{equation}\label{ConditionsClassification}
 	\begin{cases}
 	\alpha+\beta>0 ,\\
 	n_1>n_2>\dots >n_\alpha, \\
 	n_\alpha>   \ell_1\geq \ell_2\geq \dots\geq \ell_\beta \geq 1, \\
 	n_\alpha> r_1>r_2>\dots >r_\alpha \geq 0,  \\ 
 	n_1-r_1<n_2-r_2< \dots <n_\alpha-r_\alpha.
 	\end{cases}
 	\end{equation}  
 		Moreover, if two   groups of  the form \eqref{GroupClassification}  with parameters satisfying  \eqref{ConditionsClassification} are isomorphic, then they have the same lists of parameters.
  \end{theorem} 
 
  \begin{theorem}[\cite{Leong1979}]  \label{TheoremClassification2}  If $G$ is a non-trivial finite $2$-group with cyclic center and nilpotency class $2$, then either	$G$ is isomorphic to a group of the form \eqref{GroupClassification} for some non-negative integers $\alpha,\beta,n_1,r_1,\dots, n_\alpha, r_\alpha,  \ell_1,\dots ,\ell_\beta $  satisfying \eqref{ConditionsClassification} and $1<n_1-r_1$, 
 	or $G$ is isomorphic to a group of the form
 	\begin{equation}\label{GroupClassification2b}
 	R(n)*Q_1(\ell_1,\ell_1)*\dots * Q_\beta (\ell_\beta,\ell_\beta) 
 	\end{equation}
 	for  non-negative integers $\beta ,n,\ell_1,\dots, \ell_\beta$ satisfying
 	\begin{equation}\label{ConditionsClassification2b} 
 	n \geq  \ell_1\geq \ell_2\geq \dots\geq \ell_\beta \geq 1 . 
 	\end{equation}  
 	
 	Moreover, if two   of  the described groups are isomorphic, then either both groups are of the form \eqref{GroupClassification} with the same lists of parameters satisfying \eqref{ConditionsClassification} or they are both of the form  \eqref{GroupClassification2b}   with the same lists of parameters satisfying \eqref{ConditionsClassification2b}.
 \end{theorem}

 An immediate inductive argument shows that:
 \begin{lemma}\label{LemmaPowers}
 	Let $G$ be of class $2$.  Then for each $g,h\in G$ and $n\geq 1$
 	$$(gh)^{n}=g^{n}h^{n} [h,g]^{\frac{n(n-1)}{2}} .$$   
 \end{lemma}

 The following two lemmas are obvious consequences of the presentations above, and of \Cref{LemmaPowers} for the first item of each.
 \begin{lemma}
 	 	Let $G=	Q(n,r)$. Then $G$ has nilpotency class at most $2$ and
 	\begin{enumerate}
 		\item If  $p>2$ or $n>r$, then $\exp(G)=p^n$. Otherwise $\exp(G)=2^{n+1}$ . 
 		\item $G'\cong C_{p^r}$. 
 		\item $G/G'\cong C_{p^{\max\{ r,n-r\}}} \times C_{p^r}$. Here the first factor is generated by the image of $a$ in $G/G'$ and the second by the image of $b$. 
 		\item  $\mathcal Z(G) = G' \langle a^{p^r} \rangle\cong C_{p^{\max(r,n-r)}}$. When $n < 2r$ this group equals $G'$ and when $2r \leq n$, it is generated by $a^{p^r}$.
 		\item  $G/\mathcal Z(G) \cong C_{p^r}\times C_{p^r}. $  
 	\end{enumerate}
 \end{lemma}

\begin{lemma}
	Let $p=2$ and $G=R(n)$. Then $G$ has nilpotency class   $2$ and
	\begin{enumerate}
		\item $\exp(G)=2^{n+1}$.
		\item $G'=\ZZ(G)\cong C_{p^n}$. 
		\item $G/G'\cong C_{2^{n}}\times C_{2^n}$. 
	\end{enumerate}
\end{lemma}

As an immediate consequence of the definition of central product and conditions \eqref{ConditionsClassification}  and Lemma 2.4 we have

\begin{lemma}\label{GeneralRemark}
	Let $G$ be the group \eqref{GroupClassification} with the parameters satisfying \eqref{ConditionsClassification}. Then:
	\begin{enumerate}
		\item \label{ExponentGeneral} The exponent of $G$ is given by  $$\exp(G)=\begin{cases}
			 p^{n_1}, & \text{if }\alpha\neq 0; \\
			 p^{\ell_1}, & \text{if }\alpha=0\text{ and }p>2; \\
			 p^{\ell_1+1},&\text{if }\alpha=0\text{ and }p=2.
		\end{cases}$$  
		\item\label{CenterGeneral} The maximum of the orders of the centers of the factors involved in the central product is exactly the order of the center of the central product. That is,
		$$ \mathcal Z(G)\cong C_{p^\zeta},\qquad \text{where } \qquad \zeta= \max(r_1,n_\alpha-r_\alpha,\ell_1).$$
		\item\label{DerivedGeneral} The maximum of the orders of the derived subgroups of the factors involved in the central product is exactly the order of the derived subgroup of the central product. That is, 
		$$G'\cong C_{p^\delta}, \qquad \text{where } \qquad \delta= \max(r_1,  \ell_1). $$
		\item \label{QCenterGeneral}
		There is an isomorphism   $$\frac{G}{\mathcal Z(G)}\cong  \frac{Q_1(n_1,r_1)}{\ZZ(Q_1(n_1,r_1))}  \times \dots \times \frac{Q_{\alpha+\beta}(\ell_\beta,\ell_\beta )}{\ZZ(Q_{\alpha+\beta}(\ell_\beta,\ell_\beta ))} \cong \prod_{i=1}^\alpha C_{p^{r_i}}^2 \times \prod_{i=1}^\beta C_{p^{\ell_i}}^2. $$ 
		%		  \item \label{LeoIneq} $2 r_i\leq n_i$ for some $i$ implies $2r_j <n_j$ for each $j>i$. 
		%			
		%			
		%			\item The order of $G$ is given by $|G|=p^N$, with
		%			$$ N=\sum_{i=1}^{\alpha} 2r_i+ \sum_{i=1}^\beta 2\ell_i +\zeta.$$
		%			
		\item The minimal number of generators of $G$ is given by
		$$d(G) =\begin{cases}
		2(\alpha+\beta )-1, &\text{if }r_\alpha=0; \\
		2(\alpha +\beta), &\text{otherwise.}
		\end{cases} $$

	\end{enumerate}
\end{lemma}
Similarly, using now condition \eqref{ConditionsClassification2b}  and Lemma 2.5:
\begin{lemma}\label{GeneralRemark2}
	Let $G$ be the group \eqref{GroupClassification2b} with parameters satisfying condition \eqref{ConditionsClassification2b}. Then:
	\begin{enumerate}
		\item \label{ExponentGeneral2}   $\exp(G)=p^{n+1}$.
		\item\label{CenterGeneral2} 
		$  \mathcal Z(G)\cong C_{p^n}.$ 
		\item\label{DerivedGeneral1}  
		$ G'\cong C_{p^n} .$ 
		\item \label{QCenterGeneral2}
		There is an isomorphism   $$\frac{G}{\mathcal Z(G)}\cong  \frac{R(n)}{\ZZ(R(n))}\times  \frac{Q_1(\ell_1,\ell_1)}{\ZZ(Q_1(\ell_1,\ell_1))}  \times \dots \times \frac{Q_{ \beta}(\ell_\beta,\ell_\beta )}{\ZZ(Q_{ \beta}(\ell_\beta,\ell_\beta ))} \cong   C_{p^{n}}^2 \times \prod_{i=1}^\beta C_{p^{\ell_i}}^2. $$ 
		%		  \item \label{LeoIneq} $2 r_i\leq n_i$ for some $i$ implies $2r_j <n_j$ for each $j>i$. 
		%			
		%			
		%			\item The order of $G$ is given by $|G|=p^N$, with
		%			$$ N=\sum_{i=1}^{\alpha} 2r_i+ \sum_{i=1}^\beta 2\ell_i +\zeta.$$
		%			
		\item The minimal number of generators of $G$ is $2(\alpha+\beta)$.  
	\end{enumerate}
\end{lemma}

\begin{remark}\label{rem:riseparados}
	Let $G$ be the group \eqref{GroupClassification} with the parameters satisfying \eqref{ConditionsClassification} and let $1<i\leq \alpha$.	Then $r_{i-1}\neq r_i+1$. Indeed, otherwise by \eqref{ConditionsClassification} one would have that $n_{i-1}-r_i-1=n_{i-1}-r_{i-1}< n_i-r_i$, and hence the contradiction $n_i \leq n_{i-1}-1< n_i$. 
\end{remark}

\section{Proof of Theorem~\ref{TheoremMIPCyclicCentreOdd}: the groups of odd order}

Unless stated explicitly otherwise $G$ denotes a $p$-group of nilpotency class at most $2$ with cyclic center. Hence, the isomorphism type of $G$ is described in Theorems~\ref{TheoremClassification} and \ref{TheoremClassification2} and we associate to $G$ the parameters from these theorems. In particular $G$ contains elements $a_i$ and $b_i$ for $1 \leq i \leq \alpha +\beta$. Until stated otherwise, $F$ will be an arbitrary field of characteristic $p$. As the properties of being of class $2$ and having a cyclic center are determined, then any other group $H$ with $FG\cong FH$ will also be of nilpotency class $2$ with cyclic center. Thus, due to Theorems~\ref{TheoremClassification} and \ref{TheoremClassification2}, it suffices to show
% (1): that whether $G$ is of the form \eqref{GroupClassification} or of the form  \eqref{GroupClassification2b} is determined; and (2):   that all the parameters occurring in the corresponding presentation are determined.  
that
\begin{itemize}
\item[(1)] whether $G$ is of the form \eqref{GroupClassification} or of the form  \eqref{GroupClassification2b} is determined,
\item[(2)] all the parameters occurring in the corresponding presentation are determined.
\end{itemize} 

As the ideal $I(F\ZZ(G))FG$ is canonical, $FG/I(F\ZZ(G))FG\cong F[G/\ZZ(G)]$ and $G/\ZZ(G)$ is abelian, by the classical result of Deskins \cite{Deskins1956} the isomorphism type of $G/\ZZ(G)$ is determined. This is also proved in \cite[6.23 Theorem]{Sandling85}. By \Cref{GeneralRemark} and \Cref{GeneralRemark2}, we have that $$G/\ZZ(G)=\begin{cases}
	\prod_{i=1}^\alpha C_{p^{r_i}}^2\times \prod_{j=1}^\beta C_{p^{\ell_i}}^2, & \text{if $G$ is of the form \eqref{GroupClassification}};  \\
	C_{p^n}^2\times \prod_{j=1}^\beta C_{p^{\ell_j}}^2,&\text{if $G$ is of the form \eqref{GroupClassification2b}}.
\end{cases}$$
Thus the multiset 
$$L^G=\begin{cases}
	\{r_1,\dots, r_\alpha,\ell_1,\dots, \ell_\beta\}, & \text{if $G$ is of the form \eqref{GroupClassification}}; \\
	\{n,\ell_1,\dots, \ell_\beta\},&\text{if $G$ is of the form \eqref{GroupClassification2b}}
\end{cases}$$
is determined (by multiset we mean a set with multiplicities taken into account, sometimes called an unordered tuple). When we work only with one fixed group $G$ we simply write $L$ for $L^G$.

%For each $t\geq 1$, if  $p^{t-1}<s \leq p^{t}$ equation \eqref{EqJennings}    yields that 
%\begin{equation*} 
%	\M_s(G) =\begin{cases}
%		G^{p^{t}} (G')^{p^{t-1 }}, &\text{if } p^{t-1}  < s\leq 2 p^{t-1} ; \\
%		G^{p^{t}}  , & \text{if } 2p^{t-1}  < s\leq    p^{t}.
%	\end{cases}  
%\end{equation*} 
%In particular, we have that
%\begin{equation}\label{JenningsSeriesCasosUtiles}
%	\M_{p^t}(G) =G^{p^t}, \quad \M_{p^t+1}(G) =\mathcal  M_{2p^t}=G^{p^{t+1}}(G')^{p^t} \qand  \M_{2p^t+1}(G) =\M_{p^{t+1}}(G)=G^{p^{t+1}}.
%   \end{equation} 
%If   $p=2$, by \Cref{LemmaPowers} it is clear that $G^{2^{t+1}}=G^{2^{t+1}}(G')^{2^{t}}$, so in such case $\M_{p^t}(G)=G^t$ and $\M_{2^t+1}(G)=\M_{2^{t+1}}(G)=G^{p^{t+1}}$. 

\begin{lemma}\label{lemma:JenningsCapCenterDet}
	For each $s\geq 1$, the isomorphism type of $\M_s(G)\cap \ZZ(G)$ is determined. 
\end{lemma}
\begin{proof}
Since $\ZZ(G)$ is cyclic, it suffices to show that the order of $\M_s(G)\cap \ZZ(G)$ is determined.  
The ideals $I(G)^s$ and $I(\ZZ(G))FG$ are canonical by the discussion above. Moreover for each $t\geq 0$,  as  $\ZZ(G)$ is cyclic,  $I(\ZZ(G)^{p^{t}})FG  =(I(\ZZ(G))FG)^{p^{t}} $ is also canonical. Observe that 
$$I(\ZZ(G)^{p^{t}})FG \subseteq  I(G)^s \quad \text{if and only if}\quad \ZZ(G)^{p^t}\subseteq \M_s(G). $$
Indeed, if $I(\ZZ(G)^{p^{t}})FG \subseteq  I(G)^s $, taking intersections with $G-1$ we obtain that $\ZZ(G)^{p^t}-1\subseteq \M_s(G)-1$. Conversely, as $I(\mathcal Z(G)^{p^t})FG$ is the ideal generated by the elements of the form $g-1$ with $g\in \mathcal Z(G)^{p^t}$, if $\mathcal Z(G)^{p^t}-1\subseteq \M_s(G)-1 \subseteq I(G)^s$, the inclusion of ideals follows. Hence 
$$\M_s(G)\cap \ZZ(G)=\ZZ(G)^{p^{t(s)}}, $$
where $t(s)$ is the minimum integer such that $I(F\ZZ(G)^{p^{t(s)}})\subseteq I(G)^s$. Thus $t(s)$ is determined. Since the order of $\ZZ(G)$ is determined, it follows that $|\M_s(G) \cap \ZZ( G)|=|\ZZ(G)^{t(s)}|=|\ZZ(G)|/p^{t(s)}$ is determined.
\end{proof}

 % \underline{Unless stated otherwise, from now on we will assume that $G$ is of the form \eqref{GroupClassification}.}
 
 We denote 
 \[\zeta=\log_p|\ZZ(G)| \ \ \text{and} \ \ \delta=\log_p|G'|.\] 
  For each integer $t>0$  we define $\i(t)$ to be the unique index verifying  
  \[r_{\i(t )}>t\geq r_{\i(t)+1}\]
   and  $\j(t)$ to be the unique index satisfying 
   \[\ell_{\j(t)} > t \geq \ell_{\j(t)+1},\]
    where we use the conventions 
  \[r_{\alpha+1}=\ell_{\beta+1}=0, \ r_0=\ell_0=+\infty, \  n_{\alpha+1}=\zeta \ \ \text{and} \ \ n_0=0.\]
%  \ChLeo{check in the end, of these conventions are like this and really needed. In particular, $n_0=0$ is strange}
   With this conventions, $r_\alpha=+\infty$ if and only if $\alpha=0$. 
 If $G$ is of the form \eqref{GroupClassification2b} we understand $\alpha=0$.  The expression \eqref{eq:JenningsSeries} for the Jennings series takes the form
  \begin{equation*}\label{EqJennings}
  \M_{s}(G) =\prod_{ip^j\geq s} \gamma_i(G)^{p^j}= G^{p^{\lceil \log_p s\rceil }} (G')^{p^{\lceil \log_p(s/2) \rceil}},
  \end{equation*}
  where $\lceil r\rceil$ denotes the smallest integer greater than or equal to $r$, for each $r\in \R$.  Then we derive that 
  \begin{equation}\label{JenningsSeriesCasosUtiles}
  	D_{p^t}(G) =G^{p^t}.
  \end{equation} 
 For $p>2$ this is immediate; for $p=2$ it suffices to consider that $G = \langle [g,h] \rangle$ for some $g,h \in G$, as $G$ has class $2$, and that by \Cref{LemmaPowers} we have $[h,g] = h^{-2}g^{-2}(gh)^2$, i.e. $G' \leq G^2$. It is also easy to see that
 \begin{equation}\label{JenningsSeriesCasoMasUno}
 D_{p^t+1}(G)=G^{p^{t+1}} (G')^{p^t} 
 \end{equation}

We introduce the following notation from now on. Let $$\varepsilon = \begin{cases} 0, \ \ \text{if} \ \ p \neq 2, \\ 1, \ \ \text{if} \ \ p = 2. \end{cases}$$ For a positive integer $t$ we set $$
\mu(t) =   \min\{ \zeta+r_{\i(t)} -n_{\i(t)}, \zeta+ t-n_{\i(t)+1}\}, \qand
\nu(t)= \min\{\mu(t), \zeta+t-\delta-\varepsilon\}.$$

\begin{lemma}\label{PreLemmaJenningsCapCenter}
    Let $t$ be a positive integer. 
     Then
    \begin{equation*}
    G^{p^t}\cap \ZZ(G)=    \ZZ(G)^{p^{\nu(t)}}.
    \end{equation*}
If moreover $t\geq r_\alpha$, then
    \begin{equation*}
        G^{p^t}\cap \ZZ(G)=
            \ZZ(G)^{p^{\mu(t)}} .  
    \end{equation*}  

%and
%\begin{equation*}
%    \mathcal M_{p^t+1}\cap \ZZ(G)=\M_{p^{t+1}}\cap \ZZ(G).
%\end{equation*}
\end{lemma}

\begin{proof}  For $1\leq i \leq \alpha  $, let $e_i$ be the unique integer satisfying   $0\leq e_i\leq \zeta$  and
    \begin{equation*}\label{def:ei}
        \GEN{a_i}\cap \mathcal Z(G)= \ZZ(G)^{p^{e_i}} .
    \end{equation*}
    As $\GEN{a_i}\cap\mathcal Z(G)=\GEN{a_i^{p^{r_i}}}$, by  taking orders it is clear that
    $ e_i   =\zeta-n_i+r_i. $
    Note that for each $t\geq 0$ we have that
    $$ \GEN{a_i^{p^t}}\cap \mathcal Z(G)=  \ZZ(G)^{p^{\max(e_i, e_i -n_i+t)}}  =  \ZZ(G)^{p^{\max(\zeta-n_i+r_i, \zeta-n_i+t)}}  =\ZZ(G)^{p^{\zeta-n_i+\max(r_i , t)   }}  .$$
    Therefore $$\GEN{a_1^{p^t},\dots, a_\alpha^{p^t}}\cap \mathcal Z(G)=\ZZ(G)^{p^{\zeta+   \min_{i=1,\dots,\alpha} \left( \max(r_i,t)-n_i  \right)}} . $$
    Since $r_{\i(t)  }>t \geq r_{\i(t)+1}$ and
    $$r_1-n_1 > \dots> r_{\i(t   )   .
    }-n_{\i(t) }>t -n_{\i(t)}  < t-n_{\i(t)+1} < \dots <t- n _\alpha $$
    we have that
    $$\GEN{a_1^{p^t},\dots, a_\alpha^{p^t}}\cap \mathcal Z(G)=\ZZ(G)^{p^{\zeta+   \min  (r_{\i(t)} -n_{\i(t)}, t-n_{\i(t)+1})}} . $$
    Moreover, observe that  $\GEN{b_i}\cap \mathcal Z(G)=\GEN{a_j}\cap\mathcal Z(G)  =1$ for $1 \leq  i \leq \alpha+\beta$ and $1+\alpha\leq j \leq \alpha+\beta$, so that when we take into account an $a_i$ in the intersection $G^{p^t} \cap \ZZ(G)$ we do not need to consider the corresponding $b_i$ and neither do we need to consider the $a_j$ for $1+\alpha \leq j \leq \alpha + \beta$. Moreover, let $i=1$ or $i=\alpha+1$, so that the order of $[b_i,a_i]$ equals $\delta$. Then $(a_ib_i)^{p^t} = a_i^{p^t}b_i^{p^t}[b_i,a_i]^{\frac{p^t(p^t-1)}{2}}$ by Lemma~\ref{LemmaPowers}. So $[b_i,a_i]^{\frac{p^t(p^t-1)}{2}}$ is an element of $G^{p^t} \cap \ZZ(G)$ which has order $p^{\delta-t+\varepsilon}$. 	
     This yields
%     This, together with \Cref{LemmaPowers}, yields that
    \begin{equation*}%\label{eq:AgemoCapCenter}
        G^{p^{t}}\cap \mathcal Z(G)= \ZZ(G)^{p^{\zeta+\min \{r_{\i(t)} -n_{\i(t)}, t-n_{\i(t)+1}, t-\delta-\varepsilon \}}} = \ZZ(G)^{p^{\nu(t)}},
    \end{equation*}
    as desired. Now assume also that $t\geq r_\alpha$. Then $\i(t)<\alpha$ and hence $\alpha_{\i(t)+1} \geq n_\alpha$. Recall from Lemma~\ref{GeneralRemark} and the relations between the parameters that $\delta = \max\{r_1,\ell_1\} < n_\alpha$. Hence,
    $$\zeta+ \min   \{  r_{\i(t)} -n_{\i(t)}, t-n_{\i(t)+1} \} \leq \zeta+t-n_{\i(t)+1} < \zeta+ t-\delta $$  
     and the result follows as then $\mu(t) = \nu(t)$.
\end{proof}

\begin{lemma}\label{LemmaJenningsCapCenter} Let $t$ be a positive integer. 
Moreover, set $\tilde \nu(t)= \min\{ \mu(t+1),  \zeta +t-\delta\}$.
   % \[\mu = \min\{ \mu(t+1),  \zeta +t-\delta\} \ \ \text{and}  \ \ \nu =  {\ \min\{ \mu(t), \zeta+t-\delta-\varepsilon\}}\]
     Then
        \begin{equation*}
        \M_{p^t}(G)\cap \ZZ(G)=    \ZZ(G)^{p^{\nu(t)}}
    \end{equation*}
    and
    \begin{equation*}
        \M_{p^t+1}(G)\cap \ZZ(G)=\ZZ(G)^{p^{\tilde \nu (t)}}.
    \end{equation*}
If moreover  $t+1\geq r_\alpha$, then
$\M_{p^{t}+1}(G)\cap \ZZ(G)=\M_{p^{t+1}}(G)\cap \ZZ(G)$.
\end{lemma}

\begin{proof} The first identity is obvious from \Cref{PreLemmaJenningsCapCenter} and \eqref{JenningsSeriesCasosUtiles}.
    For the second one, recall that if $N,H$ and $C$ are normal subgroups of $G$, with $H\leq C$, then $(NH)\cap C=(N\cap C)H$. In particular, by \eqref{JenningsSeriesCasoMasUno} we obtain
    \begin{align*}%\label{IntersectionJenningsptPlusOneCenter}
        \M_{p^t+1}(G)\cap \ZZ(G)&=(G^{p^{t+1}} (G')^{p^t}) \cap \ZZ(G) = (G^{p^{t+1}} \cap \ZZ(G)) (G')^{p^t} \nonumber \\
        &=     \ZZ(G)^{p^{\ \min\{ \zeta+r_{\i(t+1)} -n_{\i(t+1)}, \zeta+ t+1-n_{\i(t+1)+1},  \zeta -\delta +t\}}} = \ZZ(G)^{p^{\tilde \nu(t)}}
    \end{align*}
by \Cref{PreLemmaJenningsCapCenter} and using that $G'=\ZZ(G)^{p^{\zeta-\delta}}$. 
%The last equality in particular implies that $(G')^{p^t}$ contains $\ZZ(G)^{p^{\zeta-\delta+t+1-\varepsilon}}$, so that we do not to consider $\varepsilon$ in \eqref{IntersectionJenningsptPlusOneCenter}.
Note that there is no difference here between the case of $p$ being odd or even, since the biggest order of an element of $G'$ which also lies in $G^{p^{t+1}}$ is $p^{\delta-(t+1)+\epsilon}$ which is smaller than or equal to $p^{\delta-t}$ in any case.

Assume now that $t+1\geq r_\alpha$. Then $\i(t+1)<\alpha$
%, as $t+1 = r_{\alpha-1}$ is impossible by Remark~\ref{rem:riseparados}.
and so $n_{\i(t+1)+1} \geq n_\alpha$. By the definition of $\mu(t+1)$ we trivially get $\mu(t+1) \leq \zeta+t+1 -n_{\i(t+1)+1}$, so by the previous inequality $\mu(t+1) \leq \zeta+t+1 - n_\alpha$. As $n_\alpha > \delta$ gives $\zeta+t+1 - n_\alpha \leq \zeta+t -\delta$, this implies $\tilde \nu (t)= \min\{\mu(t+1), \zeta+t-\delta \} = \mu(t+1)$. 
On the other hand by the statement of the lemma already proven in the first paragraph $\M_{p^{t+1}}(G) \cap \ZZ(G) = \ZZ(G)^{p^{\nu(t+1)}}$. As we have observed $\mu(t+1) \leq \zeta+t+1 - n_\alpha \leq \zeta + t - \delta$ already before, so $ {\nu}(t+1) = \mu(t+1) = \tilde \nu(t)$ follows and hence $\M_{p^{t}+1}(G)\cap \ZZ(G)=\M_{p^{t+1}}(G)\cap \ZZ(G)$ holds in this case.
%$$\zeta+\min\{ r_{\i(t+1)}-n_{\i(t+1) }, t+1-n_{\i(t+1)+1} \} \leq \zeta+t+1 -n_{\i(t+1)+1}\leq \zeta +t -\delta, $$
%where the first inequality is trivial and the second follows noting that $\delta < n_\alpha$ by \eqref{ConditionsClassification} and $\i(t+1) + 1 \leq \alpha$.
%Therefore in this case $\mu = \nu$ and so $\M_{p^{t}+1}(G)\cap \ZZ(G)=\M_{p^{t+1}}(G)\cap \ZZ(G)$.
\end{proof}

We will need the following general observation which is probably well known. A similar result can be found for instance in \cite[Lemma 1.15]{SalimSandlingMaximal}. 
 
\begin{lemma}\label{lem:FrobeniusMap}
Let $G$ be any finite $p$-group of nilpotency class at most $p-1$. Then the map
\[f: I(G)/I(G)^2 \rightarrow I(G)^p/I(G)^{p+1}, \ \ x \mapsto x^p  \]
is a homomorphism of rings.
\end{lemma} 
 \begin{proof}
We will consider the $p$-power map on a part of the restricted Lie-algebra associated to $FG$, namely on $A = \bigoplus_{i=1}^p I(G)^i/I(G)^{i+1}$. We write $(\cdot,\cdot)$ for the Lie bracket in $A$. Then the claim of the lemma is exactly that the $p$-power on $A$ is a homomorphism. It was already observed by Jacobson \cite[Chapter V, Section 7]{Jacobson} that this is the case if and only if $A$ is $(p-1)$-Engel, i.e. for any $x,y \in A$ we have $(x,y,...,y) = 0$, where $y$ appears $p-1$ times. For $g \in G$ write $\bar{g} = g-1$ and let $c_1,...,c_n \in G$ such that $C = \{ \bar{c}_1,...,\bar{c}_n \}$ is an $F$-basis of $I(G)/I(G)^2$. Since the Lie-bracket is bilinear writing some elements $x,y \in A$ in this basis, we see that it would be sufficient to prove that $(\bar{g}_1,...,\bar{g}_p) = 0$ for any $\bar{g}_1,...,\bar{g}_p \in C$. Note that for $g,h \in G$   by \cite[Lemma 2.2]{MS22} one has 
 \[\bar{h}\bar{g} = \bar{g}\bar{h} + (1 + \bar{g} + \bar{h} + \bar{g}\bar{h})\overline{[h,g]} \equiv \bar{g}\bar{h} + \overline{[h,g]} \bmod I(G)^3.\]
Hence we get
  \begin{align*}
  (\bar{g}_1,...,\bar{g}_p) &= (\bar{g}_1\bar{g}_2-\bar{g}_2\bar{g}_1,\bar{g}_3,...,\bar{g}_p) \equiv (\bar{g}_1\bar{g}_2- (\bar{g}_1\bar{g}_2 + \overline{[g_2,g_1]}),\bar{g}_3,...,\bar{g}_p)  \\
  &=  (- \overline{[g_2,g_1]},\bar{g}_3,...,\bar{g}_p) \equiv (\overline{[g_3,g_2,g_1]},\bar{g}_4...,\bar{g}_p) \equiv ... \equiv \overline{[g_p,...g_2,g_1]} = 0 \mod I(G)^{p+1}
  \end{align*}
  Here in the end we have used that $[g_p,...,g_1]=1$, since $G$ has class at most $p-1$
 \end{proof}

\begin{lemma}\label{ralfadet}
    The number $r_\alpha$ is determined.
\end{lemma}

\begin{proof}
    By Lemma~\ref{GeneralRemark} $r_\alpha=0$ if and only if the minimal number of generators of $G$ is odd. Thus this property is determined by $FG$. Hence we can assume that $r_\alpha>0$.
    
    For each $t>0$, one has that $a_i^{p^t} \in \ZZ(G)$ if and only if $t \geq r_i$ if and only if $i \geq \i(t)$ for any $i \leq \alpha$. The same holds for the $b_i$. Moreover, for $j \geq \alpha+1$ we have $a_i^{p^t} \in \ZZ(G)$ if and only if $t \geq \ell_{j-\alpha}$ if and only if $j-\alpha > \j(t)$. The same holds replacing $a_j$ by $b_j$. Hence, we get
    $$\Omega_t(G:\ZZ(G))=\GEN{ a_i,b_i ,a_j,b_j :    \i(t) <i\leq \alpha, \quad  \alpha + \j(t)< j \leq \alpha+\beta  }.$$
    Let $\phi:FG \rightarrow FH$ be an augmentation preserving isomorphism of $F$-algebras for $H$ some group. As discussed above $I(\Omega_t(G:\ZZ(G)))FG$ is canonical, so that
    \begin{equation}\label{eq:OmegaCan}
        \phi(I(\Omega_t(G:\ZZ(G)))FG) =I(\Omega_t(H:\ZZ(H)))FH,
    \end{equation}
     Consider the map
      $$ f_t:\frac{I(\Omega_t(G:\ZZ(G)))FG+I(G)^2}{I(G)^2}\to \frac{I(G)^{p^t}}{I(G)^{p^{t}+1}}, \qquad  x+I(G)^2\mapsto x^{p^t} +I(G)^{p^t+1}.$$

    We first claim that the fact ``$f_t$ is not the zero map'' is determined, for each $t$. This will follow from the fact that both ideals $I(\Omega_t(G:\ZZ(G)))FG$ and $I(G)$ are canonical, but we include more details for clarity. Abusing notation we denote the set of elements which map to $0$ under $f_t$ as $\ker(f_t)$. Indeed, let $\phi:FG\to FH$ be an augmentation preserving isomorphism of $F$-algebras for some group $H$. Then by \eqref{eq:OmegaCan} and using that $\phi(I(G))=I(H)$ we have a commutative diagram  
        $$\xymatrix{
         \frac{I(\Omega_t(G:\ZZ(G)))FG+I(G)^2}{I(G)^2}\ar[r]^-{f_t^G} \ar[d]_-{\tilde \phi}& \frac{I(G)^{p^t} }{I(G)^{p^t+1} } \ar[d] \\
     \frac{I(\Omega_t(H:\ZZ(H)))FH+I(H)^2}{I(H)^2}\ar[r]^-{f_t^H} & \frac{I(H)^{p^t} }{I(H)^{p^t+1} }
        }$$
    where   vertical arrows are  isomorphisms induced by $\phi$. Therefore $\tilde \phi(\ker f_t^G)=\ker f_t^H
    $, and the claim follows.

    Denote
    $$m= \begin{cases} \min \{t\in L: f_t \text{ is not the zero map}  \}, \ \ \text{if} \ \ p > 2 \\ 
                       \min \{t\in L: f_t \text{ is not the zero map and $t<\log_p(\exp(G))-1$}\}, \ \ \text{if} \ \ p=2   \end{cases}$$
         The claim yields that $m$ is determined, so in order to prove the lemma it suffices to show that $m=r_\alpha$.   If the set in the definition of $m$ is empty,  we understand that   $m= +\infty$, and hence $\alpha=0$. 
%  We hence need to show two things: first that $f_{r_\alpha}$ is not the zero map and second that $f_{\ell_j}$ is the zero map for each $\ell_j$ when $p$ is odd and for each $\ell_j < \ell_1$ when $p=2$ and $\ell_1 + 1 =  \log_p(\exp(G))$ (note Lemma~\ref{GeneralRemark}).
        
         We observe first that if $\alpha \neq 0$, by \Cref{TheoremClassification2}  and Lemma~\ref{GeneralRemark} we have that $p=2$ implies $r_\alpha\leq r_1<n_1-1=\log_p(\exp(G))-1$. Moreover, 
    \[f_{r_\alpha }(a_\alpha-1+I(G)^2)= a_\alpha^{p^{r_\alpha}}-1 +I(G)^{p^{r_\alpha}+1}=z^{p^{\zeta -n_\alpha+r_\alpha}}-1 +I(G)^{p^{r_\alpha}+1}\]
 for some generator $z$ of $\ZZ(G)$. This element is different from zero, i.e. $z^{p^{\zeta -n_\alpha+r_\alpha}}-1 \notin I(G)^{p^{r_\alpha}+1}$, as we will show now. Note that this is equivalent to $z^{p^{\zeta-n_\alpha+r_\alpha}} \notin D_{p^{r_\alpha}+1}(G)$. This will prove that $r_\alpha\geq m$.
 
          First note that $z^{p^{\zeta -n_\alpha+r_\alpha}}\notin \ZZ(G)^{p^{\zeta-n_\alpha+r_\alpha+1}}$. So showing $\ZZ(G)^{p^{\zeta-n_\alpha+r_\alpha+1}} = \ZZ(G)\cap \M_{p^{r_\alpha}+1}(G)$ will imply $z^{p^{\zeta-n_\alpha+r_\alpha}} \notin D_{p^{r_\alpha}+1}(G)$. By Lemma~\ref{LemmaJenningsCapCenter} we have $\ZZ(G) \cap \M_{p^{r_\alpha}+1}(G) = \ZZ(G) \cap \M_{p^{r_\alpha+1}}(G)$. By Remark~\ref{rem:riseparados} we have $\i(r_\alpha+1) = \alpha-1$. So with $\tilde\nu(t)$ defined as in Lemma~\ref{LemmaJenningsCapCenter} with $t = r_\alpha + 1$ we get
\[\tilde\nu(t)  = \min\{\zeta+r_{\alpha-1} - n_{\alpha-1}, \zeta + r_\alpha+1 - n_\alpha, \zeta+ r_\alpha + 1 - \delta - \varepsilon \} = \zeta + r_\alpha+1-n_\alpha, \]
where we used the two facts that $n_\alpha - r_\alpha > n_{\alpha-1} - r_{\alpha-1}$ and $n_\alpha > \delta$. So indeed by Lemma~\ref{LemmaJenningsCapCenter} we obtain         $\ZZ(G) \cap \M_{p^{r_\alpha+1}} = \ZZ(G)^{p^{\zeta-n_\alpha+r_\alpha+1}}$ which is what we needed.
         
%$           = \ZZ(G)\cap \M_{p^{r_\alpha+1}} = \ZZ(G)\cap \M_{p^{r_\alpha}+1}$ \ChDiego{(The last equality requires thinking; Maybe write it somewhere else with details)} by \Cref{LemmaJenningsCapCenter}.

Conversely, i.e. to prove that $r_\alpha \leq m$, suppose that $t\in L $ is such that $t<r_\alpha$. The last inequality implies that
\[\{a_j-1+I(G)^2, b_j-1 +I(G)^2 : \alpha <j \leq \beta \text{ and }\ell_j\leq t \}\]
 is an $F$-basis of $\frac{I(\Omega_t(G:\ZZ(G)))FG+I(G)^2}{I(G)^2}$. Fix and integer $j$ with the property that $\alpha<j\leq \beta$ and $\ell_j\leq t$. It is clear that $a_j^{p^t} = b_j ^{p^t}=1$.
 If $p$ is odd, then by Lemma~\ref{lem:FrobeniusMap} we get for any $x,y, \in F$ that
\[f_t(x(a_j-1)+y(b_j-1)+I(G)^2) =  x^{p^t}(a_j-1)^{p^t} + y^{p^t}(b_j-1)^{p^{t}}+I(G)^{p^t+1} = I(G)^{p^t+1 }.\]
 Thus in this case $f_t$ is the zero map.

% \begin{remark}\label{NewRemark}
%     \ChDiego{[To be placed somewhere else]} Observe that $\M_n(G)\cap \ZZ(G)=1$ if and only if $\M_{n}(G)=1$. Indeed, it follows from the facts that $1\neq   \M_{p^{\delta-1}+1}(G) =G^{p^{\delta}} (G')^{p^{\delta-1}}\subseteq \ZZ(G)$,  and the Jennings series is decreasing.
% \end{remark}
%  $f_t(x(a_j-1)+y(b_j-1)+I(G)^2)=f_t( (a_j^xb_j^y-1)+I(G)^2)= (a_j^xb_j^y-1)^{p^t} +I(G)^{p^t+1 }=I(G)^{p^t+1 }$ because $(a_j^xb_j^y)^{p^t}= a_j^{x p^t}b_j^{yp^t} [a_j^x,b_j^y]^{p^t(p^t-1)/2}=[a_j^x,b_j^y]^{p^t(p^t-1)/2} \in \M_{p^t+1}$ (because it is a $p^t$-power of a commutator). Thus in this case $f_t$ is the zero map.
   Now assume $p=2$ and additionally that  $t<\log_p(\exp(G))-1$ . Observe that, also for odd $p$, we have $\M_n(G)\cap \ZZ(G)=1$ if and only if $\M_{n}(G)=1$. Indeed, this follows from the facts that $1\neq   \M_{p^{\delta-1}+1}(G) =G^{p^{\delta}} (G')^{p^{\delta-1}}\subseteq \ZZ(G)$,  and the Jennings series is decreasing.
Using the general congruence 
 \[(h-1)(g-1) \equiv (g-1)(h-1) + ([h,g]-1) \bmod I(G)^3\]
following from \cite[Lemma 2.2]{MS22} and the power-map from Lemma~\ref{LemmaPowers} an easy induction argument shows that for any $x,y \in F$ we get
\begin{align*}
f_t(x(a_j-1) + y(b_j-1)) = (x(a_j-1) + y(b_j-1))^{2^t} &\equiv x^{2^t}(a_j^{2^t}-1) + y^{2^t}(b_j^{2^t}-1) + (yx)^{2^{t-1}}([b_j,a_j]^{2^{t-1}}-1) \\
                            & \equiv (yx)^{2^{t-1}}([b_j,a_j]^{2^{t-1}}-1) \mod I(G)^{2^t+1}.
\end{align*}
where on the last congruence we used that $a_j^{2^t} = b_j^{2^t} = 1$ as $t \geq \ell_j$. Assume $f_t(x(a_j-1) + y(b_j-1)) \neq 0$. This is equivalent to $[b_j,a_j]^{2^{t-1}} \not\in \M_{2^t+1}(G)$ by the previous congruences. The order of $[b_j,a_j]$ is $p^{\ell_j}$, so in particular $[b_j,a_j]^{2^{t-1}} \neq 1$ implies $t = \ell_j$ and that $[b_j,a_j]^{2^{t-1}}$ is a central involution which lies in $\ZZ(G)^{p^{\zeta-1}} \setminus  \M_{2^t+1}(G)$. This in turn implies $\ZZ(G)^{p^{\zeta-1}} \cap \M_{2^t+1}(G) = 1$ which by the observation at the beginning of this paragraph is equivalent to $\M_{2^t+1}(G)=1$. But then $t+1\geq \log_p(\exp(G))$, contradicting our hypothesis on $t$ and hence $f_t$ is indeed the zero map for all admissible values of $t$.

%    Then   $(a_j^x b_j^y)^{2^{t}}=a_j^{x 2^t}b_j^{y2^t}  [a_j^x,b_j^y]^{2^{t-1}(2^t-1)}=
%               [a_j^x,b_j^y]^{2^{t-1}(2^t-1)}$, and this belongs to $\ZZ(G)^{p^{\zeta-1}}=(G')^{p^{\delta-1}}$, as $[a_j,b_j]$ has order at most $p^{\ell_j}\leq p^t$. To show that $f_t$ is the zero map it suffices to show that $\ZZ(G)^{p^{\zeta-1}}\subseteq \M_{p^t+1}(G) $. By the observation at the beginning of this paragraph this is not the case if and only  if $\M_{p^t+1}(G)=1$. But $\M_{p^t+1}(G)=1$ implies that $t+1\geq \log_p(\exp(G))$, contradicting our hypothesis on $t$.   
\end{proof}

\begin{lemma}\label{lemma:criterionr}
	Let $t\geq 1$. Suppose that $t>r_\alpha$.  Then $t=r_{\i(t)+1}$ if and only if
	\begin{equation}\label{eq:criterionr}
		p|D_{p^{t+1}}(G) \cap \ZZ(G)| =|D_{p^t}(G) \cap \ZZ(G)|=|D_{p^{t-1}}(G) \cap \ZZ(G)|.
	\end{equation} 
\end{lemma}
\begin{proof} First note that the hypothesis implies $\alpha > 0$. Write $i:=\i(t)$. Then $t>r_\alpha$ implies $i<\alpha$.
	Assume first that $t=r_{i+1}$. Then $\i(t-1)=i+1$, and $\i(t+1)=i $  by \Cref{rem:riseparados}. Hence by \eqref{JenningsSeriesCasosUtiles} and Lemma~\ref{PreLemmaJenningsCapCenter}, which we can apply as $t > r_\alpha$:
	\begin{eqnarray*}
		D_{p^{r_{i+1}}}(G) \cap \ZZ(G)=& \ZZ(G)^{p^{\mu(r_{i+1})}} &= \ZZ(G)^{p^{\zeta +r_{i+1}-n_{i+1}}}; \\
		D_{p^{r_{i+1}-1}}(G) \cap\ZZ(G)=&\ZZ(G)^{p^{\mu(r_{i+1}-1)}}&= \ZZ(G)^{p^{\zeta+r_{i+1} -n_{i+1}}}; \\
		D_{p^{r_{i+1}+1}}(G) \cap\ZZ(G)=& \ZZ(G)^{p^{\mu(r_{i+1}+1)}}&= \ZZ(G)^{p^{\zeta+r_{i+1}+1-n_{i+1}}} .
	\end{eqnarray*} 
Here in the first and third line we used $n_{i+1}-r_{i+1} < n_i - r_i$ and in the second line $n_{i+1} > n_{i+2}$. 
%	\begin{eqnarray*}
%		D_{p^{r_{i+1}}}(G) \cap \ZZ(G)=& \ZZ(G)^{p^{\zeta+\min\{r_i-n_i, r_{i+1}-n_{i+1}  \}}} &= \ZZ(G)^{p^{\zeta +r_{i+1}-n_{i+1}}}; \\
%		D_{p^{r_{i+1}-1}}(G) \cap\ZZ(G)=&\ZZ(G)^{p^{\zeta+\min\{  r_{i+1}-n_{i+1}, r_{i+1}-1-n_{i+2}\}}}&= \ZZ(G)^{p^{\zeta+r_{i+1} -n_{i+1}}}; \\
%		D_{p^{r_{i+1}+1}}(G) \cap\ZZ(G)=& \ZZ(G)^{p^{\zeta+\min\{  r_i-n_i, r_{i+1}+1-n_{i+1}\}}}&= \ZZ(G)^{p^{\zeta+r_{i+1}+1-n_{i+1}}} .
%	\end{eqnarray*} 
 Since $\zeta\geq n_\alpha-r_\alpha > n_{i+1}-r_{i+1}-1$,  \eqref{eq:criterionr} follows.
 
Conversely assume $t>r_{i+1}$. We have that $\i(t-1)=i$. If $t=r_i-1$  then $\i(t+1)=i-1$  and by Lemma~\ref{PreLemmaJenningsCapCenter}
$$	D_{p^{t}}(G) \cap \ZZ(G)= \ZZ(G)^{p^{\mu(r_{i}-1)}} = \ZZ(G)^{p^{\zeta +r_{i}-n_{i}}}=\ZZ(G)^{p^{\mu(r_i)}}=D_{p^{t+1}}(G) \cap\ZZ(G), $$
%\begin{eqnarray*}
%	D_{p^{t}}(G) \cap \ZZ(G)=& \ZZ(G)^{p^{\zeta+\min\{r_i-n_i, r_i-1-n_{i+1}  \}}} &= \ZZ(G)^{p^{\zeta +r_{i}-n_{i}}}; \\ 
%	D_{p^{t+1}}(G) \cap\ZZ(G)=& \ZZ(G)^{p^{\zeta+\min\{  r_{i-1}-n_{i-1}, r_i-n_{i}\}}}&= \ZZ(G)^{p^{\zeta+r_{i}-n_{i}}} ;
%\end{eqnarray*}
where we have used $n_{i} > n_{i+1}$ and $n_{i}-r_i > n_{i-1}-r_{i-1}$. Thus \eqref{eq:criterionr} fails in this case. Otherwise, i.e., if $t<r_i-1$,  then $\i(t+1)=i$ and either $\mu(t)=\zeta+r_i-n_i $, in which case also $\mu(t+1) = \zeta+r_i-n_i$, and hence by Lemma~\ref{PreLemmaJenningsCapCenter}
$$D_{p^{t+1}}(G) \cap \ZZ(G)=\ZZ(G)^{p^{\zeta+r_i-n_i}}=D_{p^{t}}(G) \cap \ZZ(G), $$ 
or $\mu(t)=\zeta+t-n_{i+1}$, in which case $\mu(t-1) = \zeta + t - 1 - n_{i-1}$ and thus, also by Lemma~\ref{PreLemmaJenningsCapCenter} and since $\zeta\geq n_\alpha-r_\alpha \geq n_{i+1}-r_{i+1} >n_{i+1}-t$, $$D_{p^{t-1}}(G) \cap \ZZ(G)= \ZZ(G)^{p^{\zeta+t-1-n_{i+1}}}\subsetneq  \ZZ(G)^{p^{\zeta+t-n_{i+1}}} =D_{p^t}(G) \cap \ZZ(G).$$    Either way \eqref{eq:criterionr} fails. 
%$\min\{r_i-n_i,t-n_{i+1}\}=r_i-n_i $, and hence by \Cref{LemmaJenningsCapCenter}
%$$D_{p^{t+1}}(G) \cap \ZZ(G)=\ZZ(G)^{p^{r_i-n_i}}=D_{p^{t}}(G) \cap \ZZ(G), $$ 
%or $\min\{r_i-n_i,t-n_{i+1}\}=t-n_{i+1}   $ and thus, also by \Cref{LemmaJenningsCapCenter} and since $\zeta\geq n_\alpha-r_\alpha \geq n_{i+1}-r_{i+1} >n_{i+1}-t$, $$D_{p^{t-1}}(G) \cap \ZZ(G)= \ZZ(G)^{p^{\zeta+t-1-n_{i+1}}}\subsetneq  \ZZ(G)^{p^{\zeta+t-n_{i+1}}} =D_{p^t}(G) \cap \ZZ(G).$$    Either way \eqref{eq:criterionr} fails. 

%Finally suppose that $\alpha=0$. Then \Cref{LemmaJenningsCapCenter} yields that $D_{p^{t}}(G) \cap \ZZ(G)\subsetneq D_{p^{t-1}}(G) \cap \ZZ(G)$, so that \eqref{eq:criterionr} fails, except if $t-1\geq \zeta$  and $p>2$, or if $t-2\geq \zeta$ and $p=2$; in these cases $D_{p^{t+1}}(G) = 1 = D_{p^t}(G)$, so that \eqref{eq:criterionr} also fails. 

\end{proof}

 \begin{lemma}\label{lema:detCaso1}
If $G$ and $H$ are both of the form \eqref{GroupClassification} and $FG \cong FH$, then $G \cong H$.
\end{lemma}
\begin{proof}
    By Theorem~\ref{TheoremClassification} it suffices to show that the ordered tuple of integers $(r_1,\dots, r_\alpha,\ell_1,\dots,\ell_\beta,n_1,\dots, n_\alpha)$ is determined. The multiset $L$ is determined by the isomorphism type of $G/\ZZ(G)$ by Lemma~\ref{GeneralRemark}. Moreover, by \Cref{ralfadet} $r_\alpha$ is determined. If $r_\alpha=+\infty$ then $\alpha=0$, and the required ordered tuple is just $(\ell_1,\dots \ell_\beta)$, which are just the elements of $L$ in the natural order. Thus we can assume that $r_\alpha<+\infty$, i.e., that $\alpha>0$. By \Cref{lemma:criterionr}, for each $t>r_\alpha$, we have that $t$ appears (exactly once) in the multiset $(r_1,\dots, r_\alpha)$ if and only if condition \eqref{eq:criterionr} holds. As this condition is determined by Lemma~\ref{lemma:JenningsCapCenterDet}, we derive that the multiset $(r_1,\dots, r_\alpha)$ is determined. Now as $L$ is determined as a multiset, so is the ordered tuple $(r_1,\dots, r_\alpha,\ell_1,\dots, \ell_\beta)$.
    %is determined, it follows that also the ordered (in the natural order) list $(\ell_1,\dots, \ell_\beta)$ is determined.

    It only remains to show that the list $(n_1,\dots, n_\alpha)$ is determined. Indeed, the previous and \Cref{lemma:JenningsCapCenterDet} yield that the list of orders $(|D_{p^{r_i}}(G) \cap \ZZ(G)|)_{i=1}^\alpha$ is determined. By \eqref{JenningsSeriesCasosUtiles}, Lemma~\ref{PreLemmaJenningsCapCenter}	and the fact that $n_{i-1}-r_{i-1} < n_i -r_i$ we obtain
    $$|D_{p^{r_i}}(G) \cap \ZZ(G)|=|\ZZ(G)^{p^{\zeta+r_i-n_i}}| =p^{ n_i-r_i}, $$
    so the lemma follows as we had shown before that $r_i$ is determined.
\end{proof}

Due to Theorem~\ref{TheoremClassification} this immediately yields:
\begin{corollary}\label{cor:OddCase}
Let $p$ be an odd prime, $F$ a field of characteristic $p$ and $G$ a $p$-group of class $2$ with cyclic center. Then the isomorphism type of $G$ is determined by $FG$.
\end{corollary}

 \begin{lemma}\label{lema:detCaso2}
If $G$ and $H$ are both of the form \eqref{GroupClassification2b} and $FG \cong FH$, then $G \cong H$.
%     If $G$ is or the form \eqref{GroupClassification2b}, then isomorphism type of $G$ is determined.
 \end{lemma}
\begin{proof}
     It suffices to prove that the list $(n,\ell_1,\dots, \ell_\beta)$ is determined, but it is just $L$ ordered in the natural order, which is determined by the isomorphism type of $G/\ZZ(G)$ by Lemma~\ref{GeneralRemark2}.
\end{proof}

We record the only case left to decide, though as we will see in the following, this turns out to be more challenging than the cases handled so far:
\begin{lemma}\label{lemma:LastCaseStanding}
Let $G$ and $H$ be non-isomorphic $p$-groups of nilpotency class $2$ with cyclic center such that $FG \cong FH$. Then $p=2$ and one of the groups, say $G$, is of form \eqref{GroupClassification}, while the other is of form \eqref{GroupClassification2b}. Moreover, we have
 $$L^G=\{\ell_1^G,\dots, \ell_\beta^G\} =\{n^H,\ell_1^H,\dots, \ell_{\beta-1}^H\} = L^H $$
 for a certain natural number $\beta$.
\end{lemma}
\begin{proof}
By Lemmas \ref{lema:detCaso1} and \ref{lema:detCaso2} we get that we can assume that $G$ is of form \eqref{GroupClassification} and $H$ of form \eqref{GroupClassification2b}. Write $L^G = \{r_1^G,...,r_{\alpha_G}^G,\ell_1^G,...,\ell_{\beta_G}^G \}$ and $L^H = \{n_H,\ell_1^H,...,\ell_{\beta_H}^H \}$. We have $r_{\alpha_H}^H = \infty$, so that by Lemma~\ref{ralfadet} we also have $r_{\alpha_G}^G = \infty$ which means $\alpha_G=0$. As $G/\ZZ(G) \cong H/\ZZ(H)$ from Lemmas~\ref{GeneralRemark} and \ref{GeneralRemark2} we obtain that $\beta_G = \beta_H+1$ and defining $\beta = \beta_G$ we can write 
 $$L^G=\{\ell_1^G,\dots, \ell_\beta^G\} =\{n^H,\ell_1^H,\dots, \ell_{\beta-1}^H\} = L^H. $$
\end{proof}

\section{The groups of even order for nice fields}\label{sec:QuadForms}
As observed in Corollary~\ref{cor:OddCase} the \Cref{lema:detCaso1} solves the modular isomorphsim problem for all fields for the groups of class $2$ with cyclic center if $p$ is odd. If $p=2$, in the light of the mentioned lemma and \Cref{lema:detCaso2}, it would be sufficient to prove that the property of being of the form \eqref{GroupClassification} or of the form \eqref{GroupClassification2b} is determined. This turns out not to be possible using only known group-theoretical invariants as already observed in \cite{HS06} where the groups $Q(2,2)$ and $R(2)$ appear with their \texttt{GAP}-Ids \texttt{[64,18]} and \texttt{[64,19]}, respectively. One can also easily check using computers that the same holds for the groups $Q(2,2)*Q(1,1)$ and $R(2)*Q(1,1)$.

In this section $F$ is always going to be a field of characteristic $2$. We are not able to distinguish between the types \eqref{GroupClassification} and \eqref{GroupClassification2b} for arbitrary fields of characteristic $2$. The assumption we need is that the polynomial $X^2+X+1$ is irreducible as an element of the polynomial ring $F[X]$. 
%and we will assume this from now on.
%Our argument relies on the following easy lemma, which fails for $\F_4$.
Our argument relies on analyzing the elements mapping to $0$ under a certain power map. These power maps turn out to produce non-degenerate quadratic forms which are similar. Under our assumption on $F$  this produces a contradiction, while this is not true if the assumption is dropped. We will shortly introduce the concepts needed and prove a lemma which is probably well known to a specialist, but which we could not encounter in the literature.
 
Let $n$ be a positive integer and $V = F^{2n}$ a vector space of dimension $2n$. Let $B: V \rightarrow F$ be a non-singular quadratic form. In particular, $B(sa) = s^2B(a)$ for any $s \in F$ and $a \in V$. The \emph{polar form} $C: V \times V \rightarrow F$ associated to $B$ is a bilinear form defined as $C(v,w) = B(v+w) + B(v) + B(w)$. As $B$ is assumed to be non-singular, $C$ is non-degenerate. Two quadratic forms $B$ and $B'$ are \emph{similar}, if there exist a constant $s \in F^\times$ and a matrix $A \in \operatorname{GL}(V)$ such that $B(v) = s B'(Av)$ holds for all $v \in V$. It was shown by Arf \cite{Arf} that $V$ has a symplectic basis $a_1,...,a_n,b_1,...,b_n$ with respect to $C$, i.e. $C(a_i,b_i) = 1$ and $C(a_i,a_j) = C(b_i,b_j) = 0$ for each $1\leq i,j \leq n$, while also $C(a_i,b_j) = 0$ when $i \neq j$. Arf also showed that then up to equivalence of quadratic forms one can assume that with respect to some symplectic basis there are $x_i,y_i,z_i \in F$ for $1\leq i \leq n$ such that
\[B(s_1,...,s_n,t_1,...,t_n) = \sum_{i=1}^n x_i s_i^2 + y_it_i^2 + z_is_it_i.\]
The \emph{Arf-invariant} $\mathcal{R}(B)$ of $B$ is defined as $\sum_{i=1}^n x_iy_i$. Note that this equals $\sum_{i=1}^n B(a_i)B(b_i)$. It was shown by Arf that for equivalent quadratic forms $B$ and $B'$ one has $\mathcal{R}(B) = \mathcal{R}(B') + t$ for some $t \in \{x^2 + x \ | \ x \in F \}$. We could not encounter this for similar quadratic forms in the literature, so we include it as a lemma. 

\begin{lemma}\label{lem:ArfType}
Let $B$ and $B'$ be non-singular quadratic forms on $V = F^{2n}$ which are similar. Then  $\mathcal{R}(B) = \mathcal{R}(B') + t$ for some $t \in \{x^2 + x \ | \ x \in F \}$.
\end{lemma}
\begin{proof}
We call $\mathcal{R}(B) + \{x^2 + x \ | \ x \in F \}$ the \emph{class} of the Arf-invariant of $B$.  Let $B'(v) = s B(Av)$ for some $s \in F^\times$ and $A \in \operatorname{GL}(V)$. Choose a symplectic basis $a_1,...,a_n,b_1,...,b_n$ with respect to $B$ and denote by $C$ and $C'$ the polar forms associated to $B$ and $B'$ respectively. Then for $v,w \in V$
\[C'(v,w) = B'(v+w) + B'(v) + B'(w) =  s(B(A(v+w)) + B(Av) + B(Aw)) = sC(Av,Aw)\]
This shows that $s^{-1}A^{-1}a_1,...,s^{-1}A^{-1}a_n,A^{-1}b_1,...,A^{-1}b_n$ is a symplectic basis of $V$ with respect to $C'$. The Arf-invariant with respect to this basis is
\[\sum_{i=1}^n B'(s^{-1}A^{-1}a_i)B'(A^{-1}b_i) = \sum_{i=1}^n sB(As^{-1}A^{-1}a_i)sB(AA^{-1}b_i) =  \sum_{i=1}^n s^2B(s^{-1}a_i)B(b_i) =  \sum_{i=1}^n B(a_i)B(b_i), \]
which is the Arf-invariant of $B$ with respect to the symplectic basis $a_1,...,a_n,b_1,...,b_n$.
As the class of the Arf-invariant does not depend on the choice of a symplectic basis by \cite[Theorem (i)]{Dye}, this proves the lemma.
\end{proof}

    \begin{lemma}
         Assume that $p=2$ and $F$ is a field of characteristic $2$ such that $X^2+X+1$ is irreducible in $F[X]$. Whether $G$ is of the form  \eqref{GroupClassification2b} or not is determined.
     \end{lemma}
 
 \begin{proof}
Let $G$ and $H$ be $2$-groups of class $2$ with cyclic center such that $G$ is of the form \eqref{GroupClassification} while $H$ is of the form \eqref{GroupClassification2b} and $\phi: FG \rightarrow FH$ an isomorphism. By Lemma~\ref{lemma:LastCaseStanding} we have
 $$L^G=\{\ell_1^G,\dots, \ell_\beta^G\} =\{n^H,\ell_1^H,\dots, \ell_{\beta-1}^H\} = L^H. $$
 We define this multiset simply as $L$.
Set $t = \ell_1^G = n^H$. Moreover, let $\gamma = \max \{1 \leq i \leq \beta \ | \ \ell_i^G = t  \}$.
% and let $s$ be the size of the multiset one obtains by removing all the elements equal to $t$ from $L$.

 We consider the map 
 $$\Lambda^G:\frac{I(G)}{I(G)^2}\to \frac{I(G)^{2^{t}}}{I(G)^{2^{t}+1}}, \quad x+I(G)^2\mapsto x^{2^{t}}+I(G)^{2^{t}+1}$$
 and also define $\Lambda^H: \frac{I(H)}{I(H)^2}\to \frac{I(H)^{2^{t}}}{I(H)^{2^{t}+1}}$ analogously. We obtain the commutative diagram
  $$\xymatrix{
     \frac{I(G)}{I(G)^2} \ar[r]^-{\Lambda_G} \ar[d]_-{\phi'} & \frac{I(G)^{2^{t}}}{I(G)^{ 2^{t}+1}} \ar[d]^-{\tilde{\phi}} \\
     \frac{I(H)}{I(H)^2} \ar[r]^-{\Lambda_H} & \frac{I(H)^{2^{t}}}{I(H)^{ 2^{t}+1}}
 } $$
where both $\tilde{\phi}$ and $\phi'$ are invertible linear maps on vector spaces induced by $\phi$. We will show that the dimensions of both $\Imagen(\Lambda_G)$ and $\Imagen(\Lambda_H)$ is $1$. So, in particular, there exists a matrix $A \in \operatorname{GL}\left(I(G)/I(G)^2\right)$ and a constant $s \in F^\times$ such that $\phi'(v) = Av$ and $\tilde{\phi}(w) = sw$ for $v \in I(G)/I(G)^2$ and $w \in \Imagen(\Lambda_G)$. So, for $v \in I(G)/I(G)^2$
\begin{align}\label{eq:LambdasAreSimilar}
\Lambda_H(Av) = \Lambda_H(\phi'(v)) = \tilde{\phi}(\Lambda_G(v)) = s\Lambda_G(v).
\end{align}
  Let 
\[\alpha=\sum_{i=1}^\beta (x_i (a_i-1)+y_i(b_i-1))+I(G)^2,\]
%\[\alpha=\sum_{i=1}^\beta (x_i (a_i-1)+y_i(b_i-1))+I(G)^2=\sum_{i=1}^\beta (a_i^{x_i}b_i^{y_i}-1)+I(G)^2\]
with $x_i, y_i \in F$ for all $i$, be a generic element of $I(G)/I(G)^2$. Then by  \Cref{LemmaPowers} and the fact that $G'$ is a cyclic group of order $2^t$ we obtain, similarly as in the proof of  \Cref{ralfadet}, that
 \begin{align*}
     \Lambda^G(\alpha)&= \sum_{i=1}^\beta (x_i^{2^t}(a_i^{2^t}-1) + y_i^{2^t}(b_i^{2^t}-1) + (y_ix_i)^{2^{t-1}} ([y_i,x_i]^{2^{t-1}}-1)+I(G)^{2^{t}+1} \\
     &=\sum_{i=1}^\gamma (y_ix_i)^{2^{t-1}}([b_i,a_i]^{2^{t-1}}-1)+I(G)^{2^{t}+1}=\left( \sum_{i=1}^\gamma (x_iy_i)^{2^{t-1}}\right)([b_1,a_1]^{2^{t-1}}-1)+I(G)^{2^{t}+1}.
 \end{align*}
% \begin{align*}
%     \Lambda^G(\alpha)&=\sum_{i=1}^\beta ((a_i^{ x_i}b_i^{ y_i})^{2^{t}}-1)+I(G)^{2^{t}+1}=\sum_{i=1}^{\beta} (a_i^{2^{t}x_i}b_i^{2^{t}y_i}[b_i^{y_i}, a_i^{x_i}]^{2^{t-1}}-1)+I(G)^{2^{t}+1} \\
%     &=\sum_{i: \ell_i=t} ([b_i^{y_i},a_i^{x_i}]^{2^{t-1}}-1)+I(G)^{2^{t}+1}=\left( \sum_{i:\ell_i=t} x_iy_i\right)([b_1,a_1]^{2^{t-1}}-1)+I(G)^{2^{t}+1},
% \end{align*}
Note that we have used that when $[b_i,a_i]^{2^{t-1}}\neq 1$, it equals $[b_1,a_1]^{2^{t-1}}$ as this is the unique central involution in $G$. We conclude that $\Imagen(\Lambda_G)$ is a 1-dimensional space which we can identify with $F$ choosing the basis element $[b_1,a_1]^{2^{t-1}}-1$.
% and that we can decompose $I(G)/I(G)^2$ into a direct sum of subspaces so that $\Lambda_G$ is a non-degenerate quadratic form on the subspace $\{a_1+I(G)^2,...,a_j+I(G)^2,b_1+I(G)^2,...,b_j+I(G)^2 \}$, for $j$ the biggest integer such that $\ell_j = t$, and $\Lambda_G$ is the zero map on the complement of this subspace. 	 

      Now let $a_0,b_0$ be the generators of $R(n)$ as a subgroup of $H$. Then a generic element of $I(H)/I(H)^2$ is of the form
      $\alpha=\sum_{i=0}^{\beta-1}(x_i(a_i-1)+y_i(b_i-1)) +I(H)^2$ where $x_i,y_i \in F$ for all $i$. Then similarly as for $G$ we get
 \begin{align*}
     \Lambda^H(\alpha)&= \sum_{i=0}^{\beta-1} (x_i^{2^t}(a_i^{2^t}-1) + y_i^{2^t}(b_i^{2^t}-1) + (y_ix_i)^{2^{t-1}} ([y_i,x_i]^{2^{t-1}}-1)+I(G)^{2^{t}+1} \\
     &=x_0^{2^{t}}(a_0-1)^{2^t} + y_0^{2^{t}}(b_0-1)^{2^t} + \sum_{i=0}^{\gamma-1} (y_ix_i)^{2^{t-1}}([b_i,a_i]^{2^{t-1}}-1)+I(G)^{2^{t}+1} \\ 
     &=x_0^{2^{t}}([b_0,a_0]^{2^{t-1}}-1) + y_0^{2^{t}}([b_0,a_0]^{2^{t-1}}-1) + \left( \sum_{i=0}^{\gamma-1} (x_iy_i)^{2^{t-1}}\right)([b_0,a_0]^{2^{t-1}}-1)+I(G)^{2^{t}+1},
 \end{align*}
%       \begin{align*}
%          \Lambda^H(\alpha)&=\sum_{i=0}^{\beta-1} ((a_i^{ x_i}b_i^{ y_i})^{2^{t}}-1)+I(G)^{2^{t}+1}=\sum_{i=0}^{\beta-1} (a_i^{2^{t}x_i}b_i^{2^{t}y_i}[b_i^{y_i}, a_i^{x_i}]^{2^{t-1}}-1)+I(G)^{2^{t}+1} \\
%          &=a_0^{x_02^{t}}b_0^{y_0 2^{t}}[b_0^{y_0},a_0^{x_0}]^{2^{t-1}}-1+\sum_{i: \ell_i=t} ([b_i^{y_i},a_i^{x_i}]^{2^{t-1}}-1)+I(G)^{2^{t}+1} \\
%          &=(x_0+y_0+x_0y_0)([b_0,a_0]^{2^{t-1}}-1)+\left( \sum_{i:\ell_i=t} x_iy_i\right)([b_0,a_0]^{2^{t-1}}-1)+I(G)^{2^t+1},
%      \end{align*}
 We have used that $[b_0,a_0]^{2^{t-1}}$ is the unique central involution in $H$. Hence also $\Imagen(\Lambda_H)$ is a 1-dimensional space which we identify with $F$ choosing the basis $[b_0,a_0]^{2^{t-1}}-1$.
 % Similarly as for $G$, we can decompose $I(H)/I(H)^2$ into the direct sum of the space $\{a_0+I(H)^2,...,a_{j-1}+I(H)^2,b_0+I(H)^2,...,b_{j-1}+I(H)^2 \}$ and a complement so that $\Lambda_H$ is a non-degenerate quadratic form on the first space and the zero map on the complement. 
 
Now let $\psi$ be the inverse of the map $F \rightarrow F^{2^{t-1}}, \ x \mapsto x^{2^{t-1}}$ which exists, since the Frobenius homomorphism is injective. Define two maps $B = \psi \circ \Lambda_G$ and $B' = \psi \circ \Lambda_H$ from $V = F^{2\beta}$ to $F$. From our calculations for $\Lambda_G$ and $\Lambda_H$ we see that these maps with respect to certain bases can be explicitly written as 
\[B(x_1,y_1,...,x_\beta,y_\beta) = \sum_{i=1}^\gamma x_iy_i, \ \ \text{and} \ \  B'(x_0,y_0,...,x_{\beta-1},y_{\beta-1}) = x_0^2 + y_0^2 + \sum_{i=0}^{\gamma-1}  x_iy_i.  \]
Hence $B$ and $B'$ are both quadratic forms on $F^{2\beta}$. We observe that we can decompose $V$ for each of the two forms into a direct sum of two subspaces $V_0 = F^{2(\beta-\gamma)}$ and $V_1 = F^{2\gamma}$, such that $B$ and $B'$, respectively, are the zero map on $V_0$ and non-singular on $V_1$. Restricting to $V_1$ we obtain $\mathcal{R}(B) = 0$ while $\mathcal{R}(B')=1$. By \eqref{eq:LambdasAreSimilar} we know that $B$ and $B'$ are similar, but then by \Cref{lem:ArfType} there exists $x \in F$ so that 
\[0 = \mathcal{R}(B) = \mathcal{R}(B') + x^2 + x = x^2 + x + 1. \]
We conclude that $X^2 + X + 1$ is reducible as a polynomial in $F[X]$, contradicting our assumption on $F$.
 
%From the definition of the Arf-invariant and the concrete expressions for $\Lambda_G$ and $\Lambda_H$ one gets using the symplectic basis $\{a_1-1,...,a_\beta-1,b_1-1,...,b_\beta-1\}$ for $I(G)/I(G)^2$ that $\mathcal{R}(\Lambda_G) = 0$ and using the symplectic basis $\{a_0-1,...,a_{\beta-1}-1,b_0-1,...,b_{\beta-1}-1\}$ for $I(H)/I(H)^2$ that $\mathcal{R}(\Lambda_H) =1$. We know by the previous discussion that $\Lambda_G$ and $\Lambda_H$ are similar, so by Lemma~\ref{lem:ArfType} there must exist an $x\in F$ so that $x^2+x+1=0$. But this is impossible due to our assumption of $F$ and so $\phi$ can not exist.  
 \end{proof}

\section{The groups of even order for finite fields}\label{sec:General2Groups}
Let $m$   be a  positive integer, and $F$ the  field with $q=2^m$ elements.
We set 
\[Q = Q(n,n)*Q(n,n)*...*Q(n,n)*Q(\ell_2)*...*Q(\ell_\beta) \ \ \text{and} \ \ R =  R(n)*Q(n,n)*...*Q(n,n)*Q(\ell_2)*...*Q(\ell_\beta)\]
where the number of $n$'s is exactly $r$ and $\ell_2 < n$. This is the last case to decide by Lemma~\ref{lemma:LastCaseStanding}. When we write $G$ it means $Q$ and $R$ at the same time. 

We will use the usual notation $a_i, b_i$ and $c$. Moreover we will use capital letter versions of elements in $G$ to denote the corresponding elements of $I(FG)$, e.g. $A_1 = a_1 + 1$ or $C = c + 1$. We will us the theory of Jennings bases as explained in \cite[Section 2.3]{MS22}.

\subsection{The map $\Lambda$} We set 
\begin{align*}
\Lambda:\frac{I(G)}{I(G)^3+I(\Omega_{n-1}(G:\ZZ(G)))FG}&\to \frac{I(G)^{2^n} }{   I(G)^{2^n+2^{n-1}+1} + I(\ZZ(G))^{2^{n-1}+1}FG },  \\[1.2em] 
x+I(G)^3+I(\Omega_{n-1}(G:\ZZ(G)))FG&\mapsto x^{2^n}+ I(G)^{2^n+2^{n-1}+1} + I(\ZZ(G))^{2^{n-1}+1}FG .  
\end{align*}

We will show that $\Lambda$ is well defined  by understanding it as an $n$-step process.  To slightly lighten the notation we will write $\ZZ = \ZZ(G)$.

\begin{lemma}\label{lemma:nstep}Set
	\begin{align*}
	\Psi_1: \frac{I(G)} {I(G)^{2^0  + 2^{n-1} + 1} + I(\ZZ)FG}  &\rightarrow \frac{ I(G)^2}{ I(G)^{2+2^{n-1}+1} + I(\ZZ)^{1+1}FG} \\[1.2em] 
	x+ I(G)^{2^0  + 2^{n-1} + 1} + I(\ZZ)FG&\mapsto x^2+I(G)^{2+2^{n-1}+1} + I(\ZZ)^{1+1}FG, 	\end{align*} 
	and for $j > 1$ 
	\begin{align*}
	\Psi_j: \Imagen (\Psi_{j-1})& \rightarrow \frac{ I(G)^{2^j}}{  I(G)^{2^j+2^{n-1}+1} + I(\ZZ)^{2^{j-1}+1}FG } \\[1.2em]
	x+I(G)^{2^{j-1}+2^{n-1}+1} + I(\ZZ)^{2^{j-2}+ 1}FG & \mapsto x^2+I(G)^{2^j+2^{n-1}+1} + I(\ZZ)^{2^{j-1}+1}FG.
	\end{align*} 
	Then for each $j\geq 1$: 
	\begin{enumerate}[\rm (a)]
			\item \label{lemma:nstepa} $\Psi_j$ is well-defined
		\item\label{lemma:nstepb} $\Imagen(\Psi_j)$ is contained in the space generated linearly by $I(\ZZ)^{2^{j-1}}FG$ and products of elements of shape $A_i^{2^t}$ and $B_i^{2^t}$ for some $i$ and $t \geq j$
		\item \label{lemma:nstepc}  the map $\Psi_{j+1}$ is a homomorphism, i.e. $\Psi_{j+1}(x + y) = \Psi_{j+1}(x) + \Psi_{j+1}(y)$.
\end{enumerate}
\end{lemma} 
 \begin{proof}We prove \ref{lemma:nstepa}, \ref{lemma:nstepb} and \ref{lemma:nstepc} simultaneously by induction.

 	 Suppose $j=1$. Consider the element $x + \mu + \nu$ with $x \in I(G)$, $\mu \in I(G)^{2^{n-1}+2}$, $\nu \in I(\ZZ)FG$. Inside this proof we will write $[.,.]$ for the Lie commutator inside $FG$. Then 
 	 \[(x+\mu+\nu)^2 = x^2 + \mu^2 + \nu^2 + [x + \nu, \mu] + [x,\nu] \]
 	 We have $\mu^2, [x+\nu,\mu] \in I(G)^{2+2^{n-1}+1}$ just from the weight of  product. Moreover as $G$ has class $2$ we have $[G,G] \subseteq \ZZ(G)$ and so in the algebra situation $[r,s] \in I(\ZZ)FG$ holds for any $r,s \in FG$. Hence $\nu^2, [x,\nu] \in I(\ZZ)^2FG$, as any summand of $\nu$ carries a factor $C$ which can be put outside of the commutator. This shows that $\Psi_1$ is well-defined, i.e., \ref{lemma:nstepa}.  Now let $x \in I(G)$. If $A$ and $B$ are some Jennings basis elements, then $[A,B] \in I(\ZZ)FG$ for any Jennings basis elements $A$ and $B$ and $x^2$ has the needed shape modulo $I(\ZZ)FG$, as e.g. $(A_1B_1)^2 \equiv A_1^2B_1^2 \bmod I(\ZZ)FG$. This proves \ref{lemma:nstepb}.

 	 To prove  \ref{lemma:nstepc} we write $x \in \Imagen(\Psi_1)$ as $x = A_i^2\alpha + B_i^2\beta + C\gamma + \delta$ with $\alpha, \beta \in F$, $\gamma \in FG$ and $\delta$ a sum of elements which have only factors of shape $A_j^{2^t}$ and $B_j^{2^t}$ and such that each summand of $\delta$ has weight at least $4$. Here we only consider $A_i^2$ and $B_i^2$ for a fixed $i$ for simplicity, as for different indices these elements commute, the arguments work also with several such summands. Then \begin{align*}
x^2 &= A_i^4\alpha^2 +  B_i^4\beta^2 + C^2\gamma^2 + \delta^2 + [A_i^2\alpha, B_i^2\beta] + [A_i^2\alpha + B_i^2\beta + \delta, C\gamma] + [A_i^2\alpha + B_i^2\beta, \delta]  
\end{align*}
% 	 Now, $[A_i^2, B_i^2]\ChDiego{+I(\ZZ)^4I(G)} = a_i^{-2}b_i^{-2}a_i^2b_i^2 - 1\ChDiego{+I(\ZZ)^4I(G)} \in \ZZ(G)^4 - 1\ChDiego{+I(\ZZ)^4I(G)}$, so $[A_i^2, B_i^2] \in I(\ZZ^4)FG$.
 	 Now, 
 	 \[[A_i^2, B_i^2] \equiv a_i^{-2}b_i^{-2}a_i^2b_i^2 - 1  \mod I(\ZZ)^4I(G)\]
 	 and as $a_i^{-2}b_i^{-2}a_i^2b_i^2 \in \ZZ^4$, we conclude that $[A_i^2, B_i^2] \in I(\ZZ)^4FG$.
 	  The same argument works to show $[A_i^2 \alpha  + B_i^2 \beta, \delta] \in I(\ZZ(G))^4FG$. Note here that $A_i^2$ does not commute with a summand of $\delta$ only if $B_i^{2^t}$ is a factor of this summand for some $t \geq 1$. Moreover $[A_i^2\alpha, \gamma] \in I(\ZZ)^2FG$, as any summand of $\gamma$ which does not commute with $A_i^2$ contributes a factor $C^2$. Similarly, as any factor in a summand of $\delta$ appears with even exponent we get $[\delta, \gamma] \in I(\ZZ)^2FG$. So, $[A_i^2\alpha + B_i^2\beta + \delta, C\gamma] = C[A_i^2\alpha + B_i^2\beta + \delta, \gamma] \in I(\ZZ)^3FG$. Hence
 	 \begin{align}\label{eq:Psi2WellDefined}
 	 x^2 \equiv A_i^4\alpha^2 + B_i^4\beta^2 + C^2\gamma^2 + \delta^2 \mod I(G)^{4+2^{n-1}+1} + I(\ZZ)^{2+1}FG. 
 	 \end{align}
 	 Moreover, for $\gamma_1, \gamma_2 \in FG$ we have $(C\gamma_1 + C\gamma_2)^2 \equiv C^2\gamma_1^2 + C^2\gamma_2^2 \bmod I(\ZZ)^3FG$, as $[\gamma_1,\gamma_2] \in I(\ZZ)$. Also, for  $\delta_1, \delta_2$ sums of elements which have only factors of shape $A_j^{2^t}$ and $B_j^{2^t}$ and such that each summand of $\delta$ has weight at least $4$ we have $(\delta_1 + \delta_2)^2 \equiv \delta_1^2 + \delta_2^2 \bmod I(\ZZ)^3FG$. 
 	 It follows from \eqref{eq:Psi2WellDefined} that $\Psi_2$ is a homomorphism, i.e. the first step of \ref{lemma:nstepc}. 
 	 
% 	 To prove  \ref{lemma:nstepc} we write $x \in \Imagen(\Psi_1)$ as $x = A_i^2\alpha + B_i^2\beta + C\gamma$ with $\gamma \in FG$ and $\alpha$ and $\beta$ elements which \ChDiego{[which are linear combinations of $1$ and   elements that?]} have only factors of shape $A_j^{2^t}$ and $B_j^{2^t}$.  Then \ChDiego{[I don't see this. Shouldn't it be $(A_i^2\alpha)^2= A_i^4\alpha^2+ A_i^2[A_i^2,\alpha]\alpha$ and similarly for $B$?]}
% 	 \[x^2 = A_i^4\alpha^2 + [\alpha,A_i^2] + B_i^4\beta^2 + [\beta, B_i^2 ] + C^2\gamma^2 + [A_i^2\alpha, B_i^2\beta] + [A_i^2\alpha + B_i^2\beta, C\gamma] \]
% 	 Now, $[\alpha,A_i^2] \neq 0$ can only happen  if $B_i$ is a factor in $\alpha$. But this factor has an exponent at least $2$ and $[B_i^2, A_i^2] = [b_i^2, a_i^2] - 1 \in \ZZ(G)^4 - 1$, so $[\alpha, A_i^2] \in I(\ZZ^4)FG$. The same argument works for $[\beta, B_i^2]$ and $[A_i^2\alpha, B_i^2\beta]$. Moreover $[A_i^2\alpha, \gamma] \in I(\ZZ^2)FG$, again because every factor in $\alpha$ has exponent at least $2$, so $[A_i^2\alpha + B_i^2\beta, C\gamma] \in I(\ZZ^3)FG$. Hence
% 	 \[x^2 \equiv A_i^4\alpha^2 + B_i^4\beta^2 + C^2\gamma^2 \mod I(G)^{4+2^{n-1}+1} + I(\ZZ^{2+1})FG \]
% 	 which is exactly \ChDiego{\sout{the claim.} \ref{lemma:nstepc}. I don't see this is ''exactly'' that. Maybe add more details? And would this proof be clearer starting with $x=A_i^2\alpha + B_i^2\beta + C\gamma + \delta$, where $\alpha,\beta\in F$, $\gamma\in FG$ and $\delta $ a linear combination of products of elements of the shape $A_i^{2^t}$, $B_i^{2^t}$ with $t>1$?}
 	 
 	 Now let $j > 1$. We first show that $\Psi_j$ is well-defined. This is essentially the same argument as in the first step using the shape of $\Imagen(\Psi_{j-1})$. So consider $x + \mu + \nu$ with $x \in \Imagen(\Psi_{j-1})$ as described in \ref{lemma:nstepb}, $\mu \in I(G)^{2^{j-1}+2^{n-1}+1}$, $ \nu \in I(\ZZ)^{2^{j-2}+1}FG$, so
 	 \[(x+ \mu + \nu)^2 = x^2 + \mu^2 + \nu^2 + [x+\nu,\mu] + [x, \nu] \]
 	 Again $\mu^2, [x+\nu, \mu] \in I(G)^{2^j + 2^{n-1}+1}$ as $x, \nu \in I(G)^{2^{j-1}}$. Also $\nu^2 \in I(\ZZ)^{2^{j-1}+1}FG$ is clear. Finally if $x_1$ is a Jennings basis element in the support of $x$ such that $x_1 \in I(\ZZ)^{2^{j-2}}FG$ then $[x_1,\nu] \in I(\ZZ)^{2^{j-1}+1}FG$ follows by pulling out the factors $C$ and if $x_2$ is in the support of $x$ s.t. $x_2$ has only factors of type $A_i^{2^{j-1}}$ and $B_i^{2^{j-1}}$, then $[x_2, \nu] \in I(\ZZ)^{2^{j-1}+1}FG$ also follows.
 	 
 	 \ref{lemma:nstepb} follows from \ref{lemma:nstepc}, so we prove the latter, which is also similar as before: let $\alpha, \beta\in F$, $\gamma \in FG$ and $\delta$ a sum of elements which have only factors of shape $A_j^{2^t}$ and $B_j^{2^t}$ and such that each summand of $\delta$ has weight at least $2^j$.  Then
%\begin{align*}
% 	 (A_i^{2^{j-1}}\alpha + B_i^{2^{j-1}}\beta + C^{2^{j-2}}\gamma)^2  = & A_i^{2^j}\alpha^2 + [\alpha,A_i^{2^{j-1}}] + B_i^{2^j}\beta^2 + [\beta, B_i^{2^{j-1}} ]  \\
% 	  & + C^{2^{j-1}}\gamma^2 + [A_i^{2^{j-1}}\alpha, B_i^{2^{j-1}}\beta] + [A_i^{2^{j-1}}\alpha + B_i^{2^{j-1}}\beta, C^{j-2}\gamma] 
%\end{align*}
\begin{align*}
		(A_i^{2^{j-1}}\alpha + B_i^{2^{j-1}}\beta + C^{2^{j-2}}\gamma+\delta)^2    &= A_i^{2^j}\alpha^2  + B_i^{2^j}\beta^2 +  C^{2^{j-1}}\gamma^2  +\delta^2+ [A_i^{2^{j-1}}\alpha, B_i^{2^{j-1}}\beta] \\
		&   \  \ + [A_i^{2^{j-1}}\alpha + B_i^{2^{j-1}}\beta+\delta, C^{2^{j-2}}\gamma]+[A_i^{2^{j-1}}\alpha + B_i^{2^{j-1}}\beta ,\delta]. 
\end{align*}
 	Clearly, $[A_i^{2^{j-1}}\alpha, B_i^{2^{j-1}}\beta] \in I(\ZZ)^{2^j}FG$. Moreover 
 	\[[A_i^{2^{j-1}}\alpha + B_i^{2^{j-1}}\beta+\delta, C^{2^{j-2}}\gamma], [A_i^{2^{j-1}}\alpha + B_i^{2^{j-1}}\beta ,\delta]  \in I(\ZZ)^{2^{j-1}+2^{j-2}}FG,\]
 	 so \ref{lemma:nstepc} follows.
 \end{proof}

\begin{lemma}\label{lemma:LambdaWellDefined}
	The map $\Lambda$ is well-defined. 
\end{lemma}
\begin{proof}We set
	\begin{align*}
	\tilde \Lambda:  \	 \frac{I(G)}{ I(G)^{2^0 + 2^{n-1} + 1} + I(\ZZ)FG } &\rightarrow \frac{I(G)^{2^n} }{   I(G)^{2^n+2^{n-1}+1} + I(\ZZ)^{2^{n-1}+1}FG } \\
	x +I(G)^{2^0 + 2^{n-1} + 1} + I(\ZZ)FG &\mapsto x^{2^n}+ I(G)^{2^n+2^{n-1}+1} + I(\ZZ)^{2^{n-1}+1}FG.
	\end{align*}
	With the notation of \Cref{lemma:nstep}, the map $\tilde \Lambda$ is just the composition $\Psi_n\circ \Psi_{n-1}\circ\dots \circ \Psi_1$. Thus it is well-defined. 
	It suffices to show that the only elements of the Jennings basis which influence the result of $\tilde \Lambda$ are $A_i$, $B_i$, $A_iB_j$ where $1 \leq i,j \leq r$ and $A_iA_j$, $B_iB_j$ where $1 \leq i,j \leq r$ with $i \neq j$.
	  We observe that the only elements in $\Imagen(\Psi_1)$ which influence the result of $\Psi_n\circ \Psi_{n-1}\circ\dots \circ \Psi_2$ are $A_i^2, B_i^2$, $C$ and $A_iC$, $B_iC$ for $1 \leq i \leq r$. This follows as on an element
%	  \ChDiego{ $\delta=\sum_l \delta_l$ whose summands $\delta_i$ have only factors of the shape $A_j^{2^t}$ or $B_j^{2^t}$ ans weight at least $4$ maps to $\delta^{2^{n-1}}=\sum_l \delta_i^{2^{n-1}}$, where each summand has weight at least $2^{n+1}$. \sout{of shape $A_i^2\alpha$ with $\alpha \neq 1$ maps to $A_i^{2^n}\alpha^{2^{n-1}}$ which has weight at least $2^{n+1}$ as $\alpha$ itself has weight at least $2$.}} 
	  $\delta$ whose summands have only factors of the shape $A_j^{2^t}$ or $B_j^{2^t}$ and weight at least $4$, the $2^{n-1}$-th power map is linear. So a Jennings basis element in the expression of $\delta$ is mapped to an element of weight at least $2^{n+1}$.
	  Moreover $A_i^{2^n} = 0$, if $i > r$ as $A_i$ then has order at most $2^n$.  Furthermore, if $\gamma$ is of weight at least $2$ then $C\gamma$ maps to $C^{2^{n-1}}\gamma^{2^{n-1}}$ which has weight at least $2^{n+1}$. So the elements from $I(\ZZ)FG$ which influence the result are $C$ or of shape $A_iC$ or $B_iC$. But if $i > r$, then $A_i^{2^{n-1}} = B_i^{2^{n-1}} = 0$. The observation follows.
	  
	   Now the lemma follows easily: elements of shape $A_i^2$ or $B_i^2$ do not contribute anything influencing the final result, as their images under $\Psi_1$ are only $A_i^4$ or $B_i^4$, noting that their commutators land in $I(\ZZ)^2FG$. Neither do products of three or more different $A_i$ and $B_i$ contribute, as these only can give products of several squares as well as elements in $I(\ZZ)I(G)^2 $ under $\Psi_1$.
\end{proof} 

\begin{lemma}\label{lemma:kerLambda}
	$\ker(\Lambda)$ is exactly the set of the elements $$X = \sum_{i=1}^r\alpha_iA_i + \sum_{i=1}^r\beta_iB_i + \sum_{i,j=1}^r\gamma_{i,j}A_iB_j + \sum_{i=1}^{r-1} \sum_{j>i}^r\delta_{i,j}A_iA_j + \sum_{i=1}^{r-1}\sum_{j>i}^r\epsilon_{i,j}B_iB_j+J, $$ where $J= I(G)^{2^n+2^{n-1}+1} + I(\ZZ)^{2^{n-1}+1}FG$, and $\alpha_i, \beta_i, \gamma_{i,j}, \delta_{i,j}$ and $ \epsilon_{i,j} \in F$  satisfy  the following conditions:
	\begin{enumerate}
		\item\label{lemma:kerLambda1} For $ 1 \leq i \leq r$ we have:
\begin{align*}
\alpha_i\beta_i & = \sum_{j=1}^r \alpha_j \gamma_{i,j} + \beta_j \delta_{i,j} \ \ \ \ \hspace{1cm} (A_iC)  \\
\alpha_i\beta_i & =  \sum_{j=1}^r \beta_j \gamma_{j,i} +  \alpha_j \epsilon_{i,j}   \ \ \ \ \hspace{1cm}  (B_iC)
\end{align*}	
where the parenthesis on the right just give a name to each equation.	
		\item \label{lemma:kerLambda2} If  $G=Q$, then 	$  \sum_{i=1}^r \alpha_i \beta_i=0$.
		\item  \label{lemma:kerLambda3} If $G=R$, then $\alpha_1^2 + \beta_1^2 + \sum_{i=1}^r \alpha_i \beta_i = 0$.
	\end{enumerate}
\end{lemma}

\begin{proof} Let $X$ as above. To simplify notation later, we set $\delta_{i,i} = \epsilon_{i,i} = 0$ and $\delta_{j,i} = \delta_{i,j}$ and $\epsilon_{j,i} = \epsilon_{i,j}$ for all $i$ and $j$.
	Then, discarding the elements in $\Imagen(\Psi_1)$ which do not influence the result of $\Lambda$, as in the proof of Lemma~\ref{lemma:LambdaWellDefined}, we may write 
	\begin{align*}
	\Psi_1(X) =& \sum_{i=1}^r\alpha_{i}^2A_i^2 +  \sum_{i=1}^r\beta_{i}^2B_i^2 +  \sum_{i=1}^r\alpha_i\beta_i(C + A_iC + B_iC) +  \sum_{i,j=1}^r\gamma_{i,j}(\alpha_jA_i + \beta_i B_j)C \\
	& + \sum_{i=1}^{r-1}\sum_{j>i}^r\delta_{i,j}(\beta_iA_j + \beta_j A_i)C + \sum_{i=1}^{r-1}\sum_{j>i}^r \epsilon_{i,j}(\alpha_iB_j + \alpha_j B_i)C
	\end{align*}
	So using that $\Psi_j$ is a homomorphism  for $j > 1$ by Lemma~\ref{lemma:nstep} and that $A_i^{2^n} = B_i^{2^n} = 0$ for $i>1$ we get 
	%\begin{align*}
	%\Lambda(X) =& \sum_{i=1}^r \left( \alpha_i^{2^n}A_i^{2^n} + \beta_i^{2^n}B_i^{2^n} + \alpha_i^{2^{n-1}}\beta_i^{2^{n-1}}C^{2^{n-1}}  + A_i^{2^{n-1}}C^{2^{n-1}}\left(\alpha_i^{2^{n-1}}\beta_i^{2^{n-1}} + \sum_{j=1}^r \alpha_j^{2^{n-1}} \gamma_{i,j}^{2^{n-1}} + \sum_{j=1}^r \beta_j^{2^{n-1}} \delta_{i,j}^{2^{n-1}} \right) \right. \\
	%&\left. + B_i^{2^{n-1}}C^{2^{n-1}}\left(\alpha_i^{2^{n-1}}\beta_i^{2^{n-1}} + \sum_{j=1}^r \beta_j^{2^{n-1}} \gamma_{j,i}^{2^{n-1}} + \sum_{j=1}^r \alpha_j^{2^{n-1}} \epsilon_{i,j}^{2^{n-1}} \right) \right)
	%\end{align*}
	\begin{align*}
	\Lambda(X) =&  \alpha_1^{2^n}A_1^{2^n} + \beta_1^{2^n}B_1^{2^n} + C^{2^{n-1}} \left(\sum_{i=1}^r \alpha_i^{2^{n-1}}\beta_i^{2^{n-1}} \right) \\
	+ \sum_{i=1}^r & \left( A_i^{2^{n-1}}C^{2^{n-1}}\left(\alpha_i^{2^{n-1}}\beta_i^{2^{n-1}} + \sum_{j=1}^r \alpha_j^{2^{n-1}} \gamma_{i,j}^{2^{n-1}} + \sum_{j=1}^r \beta_j^{2^{n-1}} \delta_{i,j}^{2^{n-1}} \right) \right. \\
	&\left.  + B_i^{2^{n-1}}C^{2^{n-1}}  \left( \alpha_i^{2^{n-1}}\beta_i^{2^{n-1}} + \sum_{j=1}^r \beta_j^{2^{n-1}} \gamma_{j,i}^{2^{n-1}} + \sum_{j=1}^r \alpha_j^{2^{n-1}} \epsilon_{i,j}^{2^{n-1}} \right) \right)
	\end{align*}

	Using that the Frobenius map is injective the conditions on the coefficients for $\Lambda(X) = 0$ yield the lemma. Indeed, we obtain \eqref{lemma:kerLambda1} looking at the coefficients of $A_i^{2^{n-1}}C^{2^{n-1}}$ and $B_i^{2^{n-1}}C^{2^{n-1}}$ for $ 1\leq i \leq r$, which are the same for $Q$ and $R$, and we  obtain \eqref{lemma:kerLambda2}  and \eqref{lemma:kerLambda3} looking at the coefficient of $C^{2^{n-1}}$.  
\end{proof}
 Given a positive integer  $r$, and a subset $P$ of the polynomial ring $F[x_1,y_1,\dots, x_r,y_r]$, we denote $$V(P)=\left\lbrace  (u_1, v_1,\dots,u_r,v_r)\in F^{2r} : f(u_1,v_1,\dots ,u_r,v_r)=0 \text{ for each }f\in P\right\rbrace ,$$ the set of common zeroes of the polynomials in $P$. Consider the polynomials in $F[x_1,y_1,\dots, x_r,y_r]$ \begin{align*} 
f_r =\sum_{i=1}^r x_i y_i, \quad
g_r =x_1^2+y_1^2+f_r, \qand
h_r =\sum_{i=1}^r x_iy_i(x_i+y_i).
\end{align*}
\begin{proposition}\label{fact:Vs}
	Identify $I(G)/(I(\Omega_{n-1}(G):\ZZ(G))FG+I(G)^2)$ with $F^{2 r }$ with the basis $A_1,B_1,\dots, A_{r}, B_r $. 
	\begin{itemize}
		\item If $G=Q$ then 
		$$\frac{\ker(\Lambda^Q)+I(\Omega_{n-1}(Q :\ZZ(Q)))FQ+I(Q)^2}{I(\Omega_{n-1}(Q :\ZZ(Q)))FQ+I(Q)^2}=V(f_r,h_r). $$
		\item If $G=R$ then 
		$$\frac{\ker(\Lambda^R)+I(\Omega_{n-1}(R :\ZZ(R)))FR+I(R)^2}{I(\Omega_{n-1}(R :\ZZ(R)))FR+I(R)^2}=V(g_r,h_r). $$
	\end{itemize}                                                                                                   
\end{proposition}                                                                                                
\begin{proof}    Consider the $2r$ equations of \Cref{lemma:kerLambda}~\eqref{lemma:kerLambda1} as a system $\mathcal S$ of linear equations with variables $\gamma_{i,j}$, $\delta_{i,j}$ and $\epsilon_{i,j}$, with $1\leq i<j\leq r$, and parameters $\alpha_i$ and $\beta_i$, with $1\leq i\leq r$.  By \Cref{lemma:kerLambda}, it suffices to prove the following claim.

	\noindent\textbf{Claim}: $\mathcal S$ has a solution if and only if $ (\alpha_1,\beta_1,\dots, \alpha_r,\beta_r)$ is a zero  of $h_r$. 
	
	 It is clear that $(0,0,\dots, 0,0)$ satisfies both conditions. Thus we may assume that one of the entries of  $ (\alpha_1,\beta_1,\dots, \alpha_r,\beta_r)$ is non-zero. By the symmetry of both $h_r$ and $\mathcal S$, we may assume without loss of generality that $\alpha_1\neq 0$.  The system $\mathcal S$ has the form $v=Lw$, where
	   $v$ is a vector of size $2r$ containing $\alpha_i\beta_i$ in the $2i$-th and $(2i-1)$-th entry, i.e.  $v = (\alpha_1\beta_1, \alpha_1\beta_1, \alpha_2\beta_2, \alpha_2\beta_2,...,\alpha_r\beta_r)^T$; $L$ is a $2r \times (2r^2-r)$-matrix containing as entries the $\alpha_i$, $\beta_i$ and $0$, and $w$ a vector of size $2r^2-r$ containing as entries the variables $\gamma_{i,j}$, $\delta_{i,j}$ and $\epsilon_{i,j}$.  We will not write $w$ explicitly, but only $L$. We index the rows of $L$ by $A_iC$ and $B_iC$ in the way that each line corresponds to the respective equation in $\mathcal S$ and index the columns by the $\gamma_{i,j}$, $\delta_{i,j}$ and $\epsilon_{i,j}$. For clarification we write a subindex $S$ at a column and a subindex $T$ at a row, i.e. for example the $\gamma_{1,1}$-column is denoted $(\gamma_{1,1})_S$ and the $A_1C$-row as $(A_1C)_T$.
	
	We state all the possible non-zero elements in $L$ column by column:
	\begin{itemize}
		\item[$(\gamma_{i,j})_S:$] the non-zero entries are $\alpha_j$ in $(A_iC)_T$ and $\beta_i$ in $(B_jC)_T$ 
		\item[$(\delta_{i,j})_S:$] the non-zero entries are $\beta_j$ in $(A_iC)_T$ and $\beta_i$ in $(A_jC)_T$ 
		\item[$(\epsilon_{i,j})_S:$] the non-zero entries are $\alpha_j$ in $(B_iC)_T$ and $\alpha_i$ in $(B_jC)_T$ 
	\end{itemize}

	We claim that $L$ has rank $2r-1$: to see that the rank is at least $2r-1$ we consider all the columns containing the entry $\alpha_1$. These are $(\gamma_{i,1})_S$ for $1\leq i \leq r$ and $(\epsilon_{1,i})_S$ for $2 \leq i \leq r$, so exactly $2r-1$ columns. These are linearly independent: $(\gamma_{i,1})_S$ is the only column containing a non-zero entry in $(A_iC)_T$, namely $\alpha_1$. Excluding the columns $(\gamma_{i,1})_S$ for all $i$, the only column with a non-zero entry in $(B_iC)_T$ is $(\epsilon_{1,i})_S$, the entry again being $\alpha_1$. We next show that the rank is not $2r$, which is the number of rows in $L$. In fact we have
	\begin{align}\label{eq:LinesLinComb}
	(B_1C)_T = (A_1C)_T\frac{\beta_1}{\alpha_1} + (A_2C)_T\frac{\beta_2}{\alpha_1} + ... + (A_rC)_T\frac{\beta_r}{\alpha_1} + (B_2C)_T\frac{\alpha_2}{\alpha_1} + ... + (B_rC)_T\frac{\alpha_r}{\alpha_1}
	\end{align}
	This equation can be observed considering it column-wise: for each type of column we consider the entry in the sum and confirm that it is the same as in $(B_1C)_T$:
	\begin{itemize}
		\item in $(\gamma_{i,1})_S$ the coefficient is $\alpha_1\frac{\beta_i}{\alpha_1} = \beta_i$
		\item in $(\gamma_{i,j})_S$ for $j > 1$ the coefficient is $\alpha_j\frac{\beta_i}{\alpha_1} + \beta_i\frac{\alpha_j}{\alpha_1} = 0$
		\item in $(\delta_{i,j})_S$ the coefficient is $\beta_j\frac{\beta_i}{\alpha_1} + \beta_i\frac{\beta_j}{\alpha_1} = 0$
		\item in $(\epsilon_{i,1})_S$ the coefficient is $\alpha_1\frac{\alpha_j}{\alpha_1} = \alpha_j$
		\item in $(\epsilon_{i,j})_S$ for $j > 1$ the coefficient is $\alpha_j\frac{\alpha_i}{\alpha_1} + \alpha_i\frac{\alpha_j}{\alpha_1} = 0$
	\end{itemize}
	Overall we conclude that indeed \eqref{eq:LinesLinComb} holds.  Hence the system $v = Lw$ has solutions if and only if the same sum operation as in \eqref{eq:LinesLinComb} performed on $v$ gives an equation (as this would be the first step of the Gau\ss-algorithm). This gives 
	\[\alpha_1\beta_1 = \alpha_1\beta_1 \frac{\beta_1}{\alpha_1} +  \alpha_2\beta_2 \frac{\beta_2}{\alpha_1} + ... +  \alpha_r\beta_r \frac{\beta_r}{\alpha_1} +  \alpha_2\beta_2 \frac{\alpha_2}{\alpha_1} + ... +  \alpha_r\beta_r \frac{\alpha_r}{\alpha_1}. \]
	After reordering and multiplying with $\alpha_1$ this is equivalent to $h_r(\alpha_1,\beta_1,\dots,\alpha_r,\beta_r)=0$. This proves the claim, and thus the lemma. 
\end{proof}

                                                                               % 	For $r>2$, this arguument fails: it can be shown that $V(f_r,h_r)$ and $V(g_r,h_r)$ are both irreducible affine varieties. However, computational verifications, using {\sf GAP}, for small values of $r$ and $|F|$ show that these affine varieties have a different number of elements, so an isomorphism of $F$-algebras $\phi:FG\to FR$ cannot exist in these cases. We list these verifications in \Cref{tabla} 

\begin{proposition} \label{Prop:SufficientCondition}If $F$ is a finite field of characteristic $2$ and there is no bijective $F$-linear map $\lambda :F^{2r}\to F^{2r}$ such that $\lambda(V(f_r,h_r))=V(g_r,h_r)$, then 
$FR\not \cong FQ$.
\end{proposition}
\begin{proof}
	Suppose that $\phi:FR\to FQ$ is a isomorphism preserving the augmentation. Abusing notation, we denote by $\phi$ every map induced by $\phi$ between quotients of $I(R)$ and $I(Q)$.   Since   the diagram 
	\begin{equation*}
	\xymatrix{
	\frac{I(R)}{ I(R)^{2  + 2^{n-1}  } + Z(I)FR } \ar[rr]^{\Lambda^R} \ar[d]^{\phi }&  & \frac{I(R)^{2^n} }{   I(R)^{2^n+2^{n-1}+1} + Z(I)^{2^{n-1}+1}FR } \ar[d]^-{\phi } \\
	\frac{I(Q)}{ I(Q)^{2  + 2^{n-1}  } + Z(I)FQ } \ar[rr]^{\Lambda^Q}  &  & \frac{I(Q)^{2^n} }{   I(Q)^{2^n+2^{n-1}+1} + Z(I)^{2^{n-1}+1}FG } 
	}
	\end{equation*}
	commutes, we have that  $\phi(\ker(\Lambda^R))=\ker(\Lambda^Q)$. Thus the induced isomorphism $$\phi :\frac{I(R) }{I(\Omega_{n-1}(R):\ZZ(R))FR+I(R)^2}\to  \frac{I(Q) }{I(\Omega_{n-1}(Q):\ZZ(Q))FQ+I(Q)^2} $$  satisfies
	$$\phi \left(\frac{\ker(\Lambda^R)+I(\Omega_{n-1}(R :\ZZ(R)))FR+I(R)^2}{I(\Omega_{n-1}(R :\ZZ(R)))FR+I(R)^2}  \right)= \frac{\ker(\Lambda^Q)+I(\Omega_{n-1}(Q :\ZZ(G)))FQ+I(Q)^2}{I(\Omega_{n-1}(Q :\ZZ(Q)))FQ+I(Q)^2}.$$
	But, with the notation of \Cref{fact:Vs}, $\phi:F^{2^r}\to F^{2r}$ is $F$-linear and maps $V(f_r,h_r)$ to $V(g_r,h_r)$, in  contradiction with the hypothesis. 
\end{proof}

\begin{problem}\label{question}
	Let $F$ be a field of characteristic $2$ containing $\F_4$, the field with $4$elements, and $r\geq 2$. Prove there is no bijective $F$-linear map $\lambda :F^{2r}\to F^{2r}$ such that $\lambda(V(f_r,h_r))=V(g_r,h_r)$.
\end{problem}

If indeed the map as described in the previous problem does not exist, \Cref{Prop:SufficientCondition} yields  a   proof for \Cref{TheoremMIPCyclicCentreEvenGoodFieds} without the restriction on the field. We prove this for $r=1$ over any field and for $r=2$ over finite fields in the following lemma,  which will give the proof of \Cref{TheoremMIPCyclicCentreEvenBadFields} relying on a recent reduction result \cite{GarciLucasDelRio23}. 

\begin{lemma} \label{lem:EstrategiaDeAngelMadrileno}
	\Cref{question} holds for $r = 1$ for any field and for $r=2$, if $F$ is finite.
\end{lemma} 
\begin{proof} 
	We will show that $|V(f_r,h_r)|>|V(g_r,h_r)|$ for $r\in \{1,2\}$. For $r=1$ this is immediate, as  $\{(0,0), (1,1)\}\subseteq V(f_1,h_1)$ but $V(g_1,h_1)=\{(0,0)\}$. 
	
We will introduce a few concepts which will solve the case $r=2$ over finite fields, but might also be useful for bigger values of $r$. Let $|F| = q$. Set
\[\psi: F^2 \rightarrow F^2, \ \ (x,y) \mapsto (x^2y + xy^2, xy). \]
Then $\psi(x,y) = (0,0)$ if and only if $x = 0$ or $y = 0$ and $(\alpha, 0) \notin \text{Im}(\psi)$ for any $\alpha \in F^*$. Now consider an element $(\alpha, \beta) \in F^2$ with $\beta \neq 0$. Then 
\[\psi(x,y) = (\alpha, \beta) \Leftrightarrow x^2y+xy^2 = \alpha, \ y = \beta/x \Leftrightarrow x\beta + \beta^2/x = \alpha \Leftrightarrow \beta x^2 + \alpha x + \beta^2 = 0\]	
When $\alpha = 0$ this polynomial has a unique root and when $\alpha \neq 0$ it has either $2$ or no  roots  in $F$. Set $A^* = \{(\alpha, \beta) \in \text{Im}(\psi) \ | \ \alpha, \beta \neq 0 \}$. So $\text{Im}(\psi) = (0, F) \cup A^*$ and
\begin{align}\label{eq:Psi-1}
|\psi^{-1}(\alpha, \beta)| = \left\{ \begin{array}{llll} 2q - 1, & (\alpha, \beta) = (0,0) \\ 1, & \alpha = 0, \beta \neq 0 \\ 2, & (\alpha, \beta) \in A^* \\ 0, & \text{else} \end{array} \right. 
\end{align}
In particular, considering that each point of $F^2$ has an image under $\psi$, this gives 
\[|A^*| = \frac{1}{2}(q^2 - (2q-1) - (q-1)) = \frac{q(q-3)}{2}  + 1 .\]

Now for $(\alpha, \beta) \in F^2$ set 
\begin{align*}
 X^r_{\alpha, \beta} &= \left\{ (a_1,b_1,a_2,b_2,...,a_r,b_r) \in F^{2r} \ \mid \ \sum_{i=1}^r a_i^2b_i + a_ib_i^2 = \alpha, \ \sum_{i=1}^r a_ib_i = \beta \right\} \ \ \text{and} \\
 Y^r_{\alpha, \beta} &= \left\{ (a_1,b_1,a_2,b_2,...,a_r,b_r) \in F^{2r} \ \mid \ \sum_{i=1}^r a_i^2b_i + a_ib_i^2 = \alpha, \ a_1^2 + b_1^2 + \sum_{i=1}^r a_ib_i = \beta \right\}.
\end{align*}
In particular, $V(f_r, h_r) = X^r_{0,0}$ and $V(g_r, h_r) = Y^r_{0,0}$. Stratifying we can write  
%\begin{align*}
%  |X^r_{\alpha, \beta}| &= \sum_{(\alpha, \beta) \in F^2} |X^{r-1}_{\sum_{i=1}^{r-1} a_i^2b_i + a_ib_i^2, \ \sum_{i=1}^{r-1} a_ib_i}| \ \ \text{and} \\
% |Y^r_{\alpha, \beta}| &= \sum_{(\alpha, \beta) \in F^2} |X^{r-1}_{\sum_{i=1}^{r-1} a_i^2b_i + a_ib_i^2, \ a_1^2 + b_1^2 + \sum_{i=1}^{r-1} a_ib_i}| 
%\end{align*} 
%\begin{align*}
%	|X^r_{\alpha, \beta}| &= \sum_{(\gamma,\delta) \in F^2} |X^1_{\gamma,\delta}|\cdot |X^{r-1}_{\alpha+\gamma, \beta+\delta}| \ \ \text{and} \\
%	|Y^r_{\alpha, \beta}| &= \sum_{(\gamma,\delta) \in F^2} |Y^1_{\gamma,\delta}|\cdot  |X^{r-1}_{\alpha+\gamma,\beta+\delta}| 
%\end{align*} 
\begin{align}\label{eq:Stratification1}
  |X^r_{0, 0}| &= \sum_{(\alpha, \beta) \in F^2} | X^{r-1}_{\alpha^2\beta + \alpha\beta^2, \alpha\beta } | \ \ \text{and} \nonumber \\ 
 |Y^r_{0, 0}|  &= \sum_{(\alpha, \beta) \in F^2} | X^{r-1}_{\alpha^2\beta + \alpha\beta^2, \alpha^2 + \beta^2 + \alpha\beta } | = |Y^1_{0,0}||X^{r-1}_{0,0}| + \sum_{(\alpha, \beta) \in F^2 \setminus{(0,0)}} | X^{r-1}_{\alpha^2\beta + \alpha\beta^2, \alpha^2 + \beta^2 + \alpha\beta } |
\end{align} 
and also
\begin{align}\label{eq:Stratification2}
	|X^r_{\alpha, \beta}| &= \sum_{(\gamma,\delta) \in F^2} |X^1_{\gamma,\delta}|\cdot |X^{r-1}_{\alpha+\gamma, \beta+\delta}| \ \ \text{and} \nonumber \\ 
	|Y^r_{\alpha, \beta}| &= \sum_{(\gamma,\delta) \in F^2} |Y^1_{\gamma,\delta}|\cdot  |X^{r-1}_{\alpha+\gamma,\beta+\delta}| 
\end{align} 

We now specify to the case $r=2$. By the fact that $|Y^1_{0,0}| = 1$, \eqref{eq:Psi-1} and \eqref{eq:Stratification1} we obtain
\[|V(g_2, h_2)| = |Y^2_{0,0}| \leq 2q-1 + 2(q^2-1) = 2q^2 + 2q - 3  \]
On the other hand with \eqref{eq:Psi-1} and \eqref{eq:Stratification2} we have
\[ |X^r_{\alpha, \beta}| = (2q-1)|X^{r-1}_{0,0}| + \sum_{\beta \in F\setminus\{0\}} |X^{r-1}_{0, \beta}| + 2 \sum_{(\alpha, \beta) \in A^*} |X^{r-1}_{\alpha, \beta}|\] 
which for $r = 2$ gives 
\[|V(f_2, h_2)| = |X^2_{0,0}| = (2q-1)^2 + (q-1) + 2\cdot 2\cdot |A^*| = 6q^2 - 9q + 4. \]
We conclude that $|V(f_2,h_2)| > |V(g_2, h_2)|$ for any finite field.
\end{proof}

\begin{proof}[Proof of Theorem~\ref{TheoremMIPCyclicCentreEvenBadFields}] By Lemma~\ref{lemma:LastCaseStanding} the only case remaining open are the groups studied in this section. Proposition~\ref{Prop:SufficientCondition} shows that this translates to Problem~\ref{question} and Lemma~\ref{lem:EstrategiaDeAngelMadrileno} solves this problem in the case described in the statement of the theorem for finite fields. Finally, the main theorem of \cite{GarciLucasDelRio23} shows that we can remove the condition on the field.
\end{proof}

We also obtain some limited computational evidence for \Cref{question} for small values of $r$ and $|F|$. 
 With the help of \texttt{GAP} we observe (see \Cref{tabla}) that the cardinalities of $V(f_r,h_r)$ and $V(g_r,h_r)$ seem to be always different:
  \begin{table}[h!]
 	$$\matriz{{|l|l|cc|}
 		|F| &r  & |V(f_r,h_r)| & |V(g_r,h_r)|     \\\hline \hline 
 		&3 & 736
 		 &  352     \\
 		4 & 	4 &  9856
 		 & 6784      \\
 		& 	5 & 143872
 		 &      119296   \\  
 		 & 	6 & 2197504
 		 &      2000896   \\ \hline 
 	16	&3 &  118336
 		 & 87616      \\
 		&4 &  19588096
 		 & 18605056      \\
 	%	16 & 	4 & a &     \\
 	%	& 	5 & a &     \\ \hline 
 	%	&3 & 1 & 0    \\
 	%	64 & 	4 & a &     \\
 	%	& 	5 & a &     \\
 		\hline}$$
 	\caption{\label{tabla}  Cardinality of $V(f_r,h_r)$ and $V(g_r,h_r)$ over the field $F$.}
 \end{table} 

\textbf{Thanks:} We would like to thank \'Angel Gonz\'alez Prieto for discussions on how to attack \Cref{question}.

	\bibliographystyle{amsalpha}
\bibliography{MIP}

\end{document}